\documentclass[11pt]{article}
\usepackage{amsmath}
\usepackage{amssymb}
\usepackage{amsthm}
\usepackage{enumerate}
\usepackage{graphicx,color}

\parskip \medskipamount
\newtheorem{theorem}{Theorem}[section]
\newtheorem{lemma}[theorem]{Lemma}
\newtheorem{corollary}[theorem]{Corollary}
\newtheorem{conjecture}[theorem]{Conjecture}
\newtheorem{proposition}[theorem]{Proposition}
\newtheorem{remark}[theorem]{Remark}
\newtheorem{definition}[theorem]{Definition}

\newcommand{\pr}[1]{\operatorname{\mathbf{P}}\left[#1\right]}

\newcommand{\prcond}[2]{\operatorname{\mathbf{P}}\left[#1\;\middle\vert\;#2\right]}

\newcommand{\prsmall}[1]{\operatorname{\mathbf{P}}[#1]}
\newcommand{\E}[1]{\operatorname{\mathbf{E}}\left[#1\right]}

\newcommand{\Econd}[2]{\operatorname{\mathbf{E}}\left[#1\;\middle\vert\;#2\right]}

\newcommand{\vol}[1]{\operatorname{vol}\left(#1\right)}

\newcommand{\critical}{\mathrm{c}}
\newcommand{\compl}{\textrm{c}}
\newcommand{\tperc}{T_\mathrm{perc}}
\newcommand{\tdet}{T_\mathrm{det}}
\newcommand{\tbroad}{T_\mathrm{bc}}
\newcommand{\taucov}{T_\mathrm{cov}}
\newcommand{\tcov}{\E{\taucov}}
\newcommand{\1}{{\text{\Large $\mathfrak 1$}}}

\newcommand{\IGNORE}[1]{}
\newcommand{\topic}[1]{\par\medskip\noindent{\bf #1}\nopagebreak\newline\noindent}

\def\rset{{\mathbb{R}}}
\def\R{{\rset}}

\def\B{{B}}

\begin{document}

\title{Mobile Geometric Graphs: Detection, Coverage and Percolation}

\author{Yuval~Peres\thanks{Microsoft Research, Redmond WA, U.S.A.\ \ Email:
        \hbox{peres@microsoft.com}.}
\and
        Alistair~Sinclair\thanks{Computer Science Division, University
        of California, Berkeley CA~94720-1776, U.S.A.\ \ Email:
        \hbox{sinclair@cs.berkeley.edu}.
        Supported in part by NSF grant CCF-0635153
        and by a UC Berkeley Chancellor's Professorship.}
\and
        Perla~Sousi\thanks{University of Cambridge, Cambridge, U.K.\ \ Email:
        \hbox{p.sousi@statslab.cam.ac.uk}.}
\and
        Alexandre~Stauffer\thanks{Computer Science Division, University
        of California, Berkeley CA, U.S.A.\ \ Email:
        \hbox{stauffer@cs.berkeley.edu}.
        Supported by a Fulbright/CAPES scholarship and NSF grants CCF-0635153 and
        DMS-0528488. Part of this work was done while the author was doing a summer
        internship at Microsoft Research, Redmond WA.}
}
\maketitle
\thispagestyle{empty}

\begin{abstract}
We consider the following dynamic Boolean model introduced by van den Berg, Meester and White (1997).
At time $0$, let the nodes of the graph be a Poisson point process in $\R^d$ with constant intensity and
let each node move independently according to Brownian motion.
At any time $t$, we put an edge between every pair of nodes whose distance is at most $r$.
We study three fundamental problems in this model: detection (the time until a target point---fixed or moving---is
within distance $r$ of some node of the graph); coverage (the time until all points inside
a finite box are detected by the graph);
and percolation (the time until a given node belongs to the infinite connected component of the graph).
We  obtain precise asymptotics for these quantities by combining ideas from stochastic geometry,
coupling and multi-scale analysis.
\newline
MSC 2010 \emph{subject classifications.}
Primary 82C43; 
Secondary 60G55, 
          60D05, 
          60J65, 
          60K35, 
          82C21. 
\end{abstract}

\setcounter{footnote}{0}
\setcounter{page}{1}
\section{Introduction}
In a \emph{random geometric graph}, nodes are distributed according to a Poisson point process
in~$\R^d$ of intensity~$\lambda$ (i.e., nodes are uniformly distributed, with $\lambda$ nodes
in expectation per unit volume), and an edge is placed between all pairs of nodes whose
distance is at most~$r$.
These graphs, also known as ``continuum percolation'' or the ''Boolean model,`` have been used extensively as models for communication networks \cite{GK2,MR,Penrose}.
While simple, these models capture important qualitative features of real networks, such as the phase transition in connectivity as the intensity $\lambda$ of the
Poisson point process is increased.

In many applications, the nodes of the network are not fixed in space but mobile.
It is natural to model the movement of the nodes by independent Brownian motions
\cite{vandenberg,Kesidis,Konst}; we call this the {\it mobile geometric graph\/} model.
We focus on three fundamental problems in this model: \emph{detection} (the time
until a target point---which may be fixed or moving---comes within distance~$r$ of a node of the graph);
\emph{coverage} (the time until all points in a large set are detected);
and \emph{percolation} (the time until a given node is connected to the infinite connected component).

We now give a precise definition of the model and state our results.
Motivation and related work are addressed in Section~\ref{sec:motivation}.

\topic{The Mobile Geometric Graph model}\\[-5pt]
Let $\Pi_0=\{X_i\}_i$ be a Poisson point process on $\rset^d$ of intensity $\lambda$.
To avoid ambiguity, we refer to the points of a point process as \emph{nodes}.
We let each node $X_i$ move according to a standard Brownian motion $(\zeta_i(s))_{s\geq 0}$ independently of the other nodes,
and set $\Pi_s = \{X_i + \zeta_i(s)\}_i$ to be the point process obtained after the nodes of $\Pi_0$ have moved for time $s$.
By standard arguments~\cite{vandenberg} it follows that $\Pi_s$ is again a Poisson point process of the same intensity~$\lambda$.

At any given time $s$ we construct a graph $G_s$ by putting an edge between any two nodes of $\Pi_s$ that are at distance at most~$r$.
In what follows we take~$r$ to be an arbitrary but fixed constant.
There exists a critical intensity $\lambda_\critical=\lambda_\critical(d)$ such that if $\lambda> \lambda_\critical$,
then a.s.\ there exists a unique infinite connected component in~$G_s$, which we denote by $C_{\infty}(s)$,
while if $\lambda < \lambda_\critical$ then all connected components are finite a.s.~\cite{MR,Penrose}.


\topic{Statement of results}\\[-5pt]
{\bf Detection.}\ \
Consider a ``target'' particle $u$ which is initially placed at the origin, and whose position at
time $s$ is given by $g(s)$, which is a continuous process in $\rset^d$.
We are interested in the time it takes for $u$ to be \emph{detected} by the mobile geometric graph,
namely how long it takes until a node of $\Pi_0$ has come within distance at most $r$ from~$u$.
More formally we define
\[
\tdet= \inf\Big\{t\geq 0: g(t) \in \bigcup_{i} B(X_i+\zeta_i(t),r)\Big\},
\]
where the union is taken over all the nodes $\{X_i\}$ of $\Pi_0$ and $B(x,r)$ denotes the ball of radius $r$ centered at $x$.
In Section~\ref{sec:detection} we prove the following theorem, which
extends previous classical results on detection
for non-mobile particles~$u$~\cite{SKM95,Kesidis}.

\begin{theorem}\label{thm:detection}
In two dimensions, for any fixed~$\lambda$ and any $g$ independent of the motions of the nodes of $\Pi_0$, we have that
\[
\pr{\tdet > t} \leq \exp\left(-2\pi \lambda \frac{t}{\log t}(1+o(1))\right).
\]
In addition, if $g$ is an independent Brownian motion then the above bound is tight, i.e.,
\begin{equation}
\pr{\tdet > t}= \exp\left(-2\pi \lambda \frac{t}{\log t}(1+o(1))\right).
\label{eq:detbm}
\end{equation}
\end{theorem}

\begin{remark}\label{rem:highdim}
\rm{
Theorem~\ref{thm:detection} is stated for $d=2$. In Section~\ref{sec:detection} it is extended to all dimensions $d\geq 1$ (see Theorem~\ref{thm:detection2}),
where the tail of $\tdet$ is shown to be $\exp(-\Theta(\sqrt{t}))$ for $d=1$ and $\exp(-\Theta(t))$ for $d\geq 3$.
All these results exploit a connection between the detection time and the volume of the Wiener sausage, as explained in Section~\ref{sec:detection}.
}
\end{remark}

\begin{remark}\rm{
From \cite{SKM95,Kesidis}, we have that for $d=2$ the result in \eqref{eq:detbm} also holds when $u$ does \emph{not} move (i.e., $g\equiv 0$). Thus
Theorem~\ref{thm:detection} establishes that, asymptotically, the best strategy for a particle $u$ (that is not informed of the motion of the nodes
of $\Pi_0$) to avoid detection is to stay put. In Section~\ref{sec:detection} we show that this is true for $d=1$ for any fixed $t$ (not only asymptotically)
and conjecture that this holds in higher dimensions as well (see Conjecture~\ref{conj:wiener}).
}\end{remark}

\par\medskip\noindent
{\bf Coverage.}\ \ Let $A$ be a subset of $\rset^d$. We are interested in the time it takes for \emph{all} the points of $A$ to be detected. We thus define
\[
\taucov(A) = \inf\Big\{t \geq 0: A \subset \bigcup_{s \leq t}\bigcup_{i}B(X_i + \zeta_i(s),r) \Big\}.
\]
For $R \in \rset_+$, let $Q_R$ be the cube in $\R^d$ of side length $R$.
A natural question proposed by Konstantopoulos~\cite{Konst} is to determine the asymptotics of $\E{\taucov(Q_R)}$ as $R \to \infty$.

%

In Section~\ref{sec:coverage} we prove Theorem~\ref{thm:coverage} below,
which gives the asymptotics for the expected time to cover the set $Q_R$ as $R \to \infty$ and shows that $\taucov(Q_R)$ is concentrated around its expectation.
We write $f \sim g$ as $x \to \infty$, when $\frac{f(x)}{g(x)} \to 1$ as $x \to \infty$.

\begin{theorem}\label{thm:coverage}
We have that as $R\to \infty$
\[
E{\taucov(Q_R)} \sim \left\{\begin{array}{rl}
   \frac{\pi}{8\lambda^2}(\log R)^2& \text{for $d=1$}\\
   \frac{1}{\lambda \pi }\log{R} \log\log{R}  & \text{for $d=2$}\\
   \frac{d\log R}{\lambda c(d) r^{d-2}}  & \text{for $d\geq 3$}
   \end{array}\right. \quad \text{ and }\quad  \frac{\taucov(Q_R)}{\E{\taucov(Q_R)}} \to 1 \text{ in probability},
\]
where $c(d) = \frac{\Gamma\left(\frac{d}{2}-1 \right)}{2 \pi^{\frac{d}{2}}}$ and $\Gamma$~stands for the Gamma function.
\end{theorem}
\begin{remark}
\rm{
Instead of covering a whole cube, we could ask for the coverage time of other sets. We prove Theorem~\ref{thm:coverage} in this general setting in
Section~\ref{sec:coverage}; for instance, we show that $\tcov$  for a line segment of length $R$ is smaller than $\E{\taucov(Q_R)}$ by a factor of
$\frac{1+o(1)}{d}$ and also obtain asymptotics for fractal sets (see Theorem~\ref{thm:coverage2}).
}
\end{remark}

\par\medskip\noindent
{\bf Percolation.}\ \
Let $u$ be an extra node initially at the origin and which moves independently of the nodes of $\Pi_0$ according to some function $g$.
We now investigate the time it takes until $u$ belongs to the infinite connected component. We denote this time by $\tperc$, which can be more formally written as
\[
\tperc = \inf\{t\geq 0: \exists y \in C_{\infty}(t) \text{ s.t. } \|g(t)-y\|_2\leq r \}.
\]

The detection time clearly provides a lower bound on the percolation time,
so we may deduce from Theorem~\ref{thm:detection} and Remark~\ref{rem:highdim} above that
$\pr{\tperc > t}$ is at least $\exp\left(-O(t/\log t)\right)$ for $d=2$
and at least $\exp\left(-O(t)\right)$ for $d \geq 3$, when $u$ is non-mobile or moves according to an independent Brownian motion. We will
prove the following stretched exponential upper bound in all dimensions
$d\geq 2$ in Section~\ref{sec:percolation}:

\begin{theorem}\label{thm:percolation}
   For all dimensions $d\geq 2$, if $\lambda>\lambda_\critical(d)$ then there exist constants $c$ and $t_0$, depending only on~$d$, such that
   $$
      \pr{\tperc > t} \leq \exp\left(-c\frac{\lambda t}{\log^{3+6/d}t}\right), \text{ for all } t\geq t_0.
   $$
    This holds when $u$ is non-mobile or moves according to an independent Brownian motion.
\end{theorem}

Theorem~\ref{thm:percolation} is the main technical contribution of the paper;
we briefly mention some of the ideas used in the proof.  The key technical
challenge is the dependency of the $G_s$'s over time.  To overcome
this, we partition~$\rset^d$ into subregions of suitable size and
show via a multi-scale argument that all such subregions contain sufficiently
many nodes for a large fraction of the time steps. This is the content of Proposition~\ref{pro:density}
which we believe is of independent interest. This result allows us to
{\it couple\/}
the evolution of the nodes in each subregion with those of a fresh Poisson
point process of slightly smaller intensity $\lambda'<\lambda$ which is still
larger than the critical value~$\lambda_\critical$.  After
a number of steps~$\Delta$ that depends on the size  of the subregion,
we are able to guarantee that the coupled processes match up almost
completely.  As a result, we can conclude that
there are $\Theta(t/\Delta)$ time steps for which the mobile geometric graph contains
an \emph{independent}
Poisson point process with intensity~$\lambda''>\lambda_\critical$.
This fact, which we believe is of wider applicability, is formally stated
in Proposition~\ref{pro:ss}.  This
independence is sufficient to complete the proof.

Finally, to illustrate a sample application of Theorem~\ref{thm:percolation}, we consider
the time taken to broadcast a message in a network of finite size.
Consider a mobile geometric graph in a cube of volume~$n/\lambda$ (so the expected\footnote{The
result can be adapted to the case of a {\it fixed\/} number of nodes~$n$ using
standard ``de-Poissonization" arguments~\cite{Penrose}.  See the Remark
following the proof of Corollary~\ref{cor:broadcast} in Section~\ref{sec:broadcast}.}
number of nodes is~$n$).  Since the volume is finite, we need to modify
the motion of the nodes to take account of boundary effects: following
standard practice, we do this by turning the cube into a torus (so that nodes
``wrap around" when they reach the boundaries).
Suppose a message originates at an arbitrary node
at time~0, and at each integer time step~$t$ each node that has already
received the message broadcasts it to all nodes in the same connected
component.  (Here we are making the reasonable assumption
that the speed of transmission is much faster than the motion of the nodes,
so that messages can travel throughout a connected component before it is
altered by the motion.)  Let $\tbroad$ denote the time until all nodes have
received the message. We prove the following result in Section~\ref{sec:broadcast}.
\begin{corollary}\label{cor:broadcast}
In a mobile geometric graph on the torus of volume~$n/\lambda$ with any fixed $\lambda > \lambda_\critical$,
the broadcast time $\tbroad$
is $O(\log n (\log \log n)^{3+6/d})$ w.h.p.\ in any dimension $d\ge 2$.
\end{corollary}

\section{Motivation and related work}\label{sec:motivation}
A possible application of our results is in the study of mobile ad hoc networks, where
nodes moving in space
cooperate to relay packets on behalf of other nodes without any centralized
infrastructure.  This is the case, for example, in vehicular
networks (where sensors are attached to cars, buses or taxis), surveillance
and disaster recovery applications (where mobile sensors are used to
survey an area), and pocket-switched networks based on mobile communication
devices such as cellphones.

Random geometric graphs have long been used as a model for {\it static\/}
wireless networks, and by now their structural and algorithmic
properties are rather well understood mathematically.
We refer the reader to the
book~\cite{Penrose} for extensive background on random geometric graphs and
to~\cite{SS10} for a selective survey of applications.  We should mention
that in many papers the model is defined over a torus $S\subset \R^d$
of finite volume~$n/\lambda$, so that the expected number of nodes is~$n$.
We choose to work in the infinite volume~$\rset^d$, which is mathematically
cleaner, but most results obtained
there can be adapted to finite volumes with a little technical work; see
Section~\ref{sec:broadcast} for an example.  In the finite setting it makes
sense to talk about the random geometric graph being connected, which is an important property
when the nodes are static.  This occurs at a sharp threshold value of~$\lambda$,
which however grows with~$n$~\cite{GK1,Penrose1,Penrose2}.

The scope of mathematically rigorous work with mobile nodes
is much more limited, and there is as yet no widespread agreement on
an appropriate model for node mobility.  The model
we use in this paper is equivalent to the ``dynamic boolean model"
introduced by Van den Berg et al.~\cite{vandenberg} (who also
proved that almost surely an infinite component exists at \emph{all} times
if $\lambda>\lambda_\critical$).  We point out that
in this model, in contrast to many others, node mobility is fixed and does
not depend on~$n$ (the number of nodes in a finite network).


The \emph{detection time} was addressed
by Liu et al.~\cite{Liu}, assuming that each node moves continuously
in a {\it fixed\/} randomly chosen direction; they show that the time it takes for
the network to detect a target is exponentially distributed with expectation
depending on the intensity~$\lambda$.  Also, for the special case of
a stationary target, as observed in~\cite{Kesidis,Konst} the detection time
can be deduced from
classical results on continuum percolation: namely, in this case it follows
from~\cite{SKM95} that
$\mathbf{Pr}[\tdet \geq t] = \exp(-\lambda \E{\vol{W_r(t)}})$, where $W_r(t)$ is the
``Wiener sausage" of radius~$r$ up to time~$t$.  This
volume in turn is known quite precisely~\cite{Spitzer,BMS}.
In a different model, in which nodes perform continuous-time random walks
on the square lattice, Moreau et al.~\cite{MOBC} (see also~\cite{DGRS})
establish that the best strategy for a target to avoid detection is
to stay put. (In these papers, a target is said to be detected when there exists a
node of the graph at the \emph{same} lattice location as the target.)
The analysis of the \emph{coverage time} was suggested as an open problem
by Konstantopoulos~\cite{Konst}.

The {\it percolation time\/} was first studied in~\cite{SS10}, which also examined
detection for moving targets.
This paper is a strengthened version of~\cite{SS10}, with tighter results on
percolation and detection as well as the addition of coverage.

The question of \emph{broadcasting} in mobile graphs has been studied by several authors
in a setting where messages travel only to immediate neighbors at each time step.
Clementi et al.~\cite{clementi} establish tight bounds for the broadcast time in
this setting assuming that either the intensity $\lambda$ or
the range of motion of the nodes grows with~$n$.
The case of smaller intensities and bounded range of motion was studied by
Pettarin et al.~\cite{PPPU10}.
A problem similar to broadcast was studied by
Kesten and Sidoravicius~\cite{KestenSidoravicius}, who derived the rate at which an infection
spreads through nodes that are performing continuous-time random walks on the square lattice.


\section{Detection time}\label{sec:detection}
In this section we give the proof of Theorem~\ref{thm:detection}.
We first state a generalization of a well-known result~\cite{SKM95},
which we will use in several proofs; we include its proof here for the sake of completeness.
\begin{lemma}\label{lem:detection}
Suppose that $u$ starts from the origin at time $0$ and its position at time $s$ is given by a deterministic function $g(s)$.
Let $W_g(t) = \cup_{s \leq t}B(g(s)-\zeta(s),r)$ be the so-called ``Wiener sausage with drift'' up to time $t$. Then,
for any dimension $d\geq1$,
the detection probability satisfies
\[
\pr{\tdet > t} =\exp(- \lambda \E{\text{\rm{vol}}(W_g(t))}),
\]
where $\vol{A}$ stands for the Lebesgue measure of the set $A$ in $\R^d$.
\end{lemma}
\begin{proof}
Let $\Phi$ be the set of points of $\Pi_0$ that have detected $u$ by time $t$, that is
\[
\Phi = \{X_i \in \Pi_0:  \exists s \leq t \text{ s.t. } g(s) \in B(X_i + \zeta_i(s), r)\}.
\]
Since the $\zeta_i$'s are independent we have that $\Phi$ is a thinned Poisson point process with intensity given by
\begin{align*}
\Lambda(x) = \lambda \pr{ x \in \cup_{s \leq t} B(g(s)-\zeta(s),r)},
\end{align*}
where $\zeta$ is a standard Brownian motion.

So for the probability that the detection time is greater than $t$ we have that
\begin{align*}
\pr{\tdet > t}
= \exp(-\lambda \int_{\rset^d} \pr{x \in \cup_{s \leq t}B(g(s)-\zeta(s),r)} dx)
\\ =\exp(-\lambda \E{\text{vol}(\cup_{s \leq t}B(g(s)-\zeta(s),r))}) = \exp(-\lambda \E{\vol{W_g(t)}}).
\end{align*}
\end{proof}

\begin{remark}\label{rem:moving}
\rm{
The preceding lemma implies that
when the motion $g$ of $u$ is random and independent of the motions of the nodes of the Poisson point process $\Pi_0$ then
\[
\pr{\tdet > t} =\E{\exp(- \lambda \E{\text{\rm{vol}}(W_g(t))\mid(g(s))_{s \leq t}})}.
\]
}
\end{remark}

\begin{remark}\label{rem:comp}
\rm{We note that the above proof can be easily generalized to show that the time $\tdet(K)$ until we detect some
 point in a compact set $K \subset \R^d$ satisfies
\begin{equation}\label{detcompact}
\pr{\tdet(K) > t} = \exp(-\lambda \E{\vol{\cup_{s \leq t}(K^r - \zeta(s))}}),
\end{equation}
where $K^r$ stands for the $r$-enlargement of $K$, i.e. $K^r = \cup_{x \in K} B(x,r)$.
}
\end{remark}

From Lemma~\ref{lem:detection} we see that estimating $\pr{\tdet > t}$ translates to deriving estimates for $\E{\vol{W_g(t)}}$.
When $u$ does not move (i.e., $g\equiv 0$), it is well known that in two dimensions~\cite{Spitzer,BMS}
$$
   \E{\vol{W_0(t)}} = \frac{2\pi t}{\log t}(1+o(1)).
$$
The following Lemma implies that $\E{\vol{W_g(t)}}\geq \E{\vol{W_0(t)}}(1-o(1))$ for any deterministic continuous $g$.

\begin{lemma}\label{lem:voldrift}
Let $\zeta$ be a standard Brownian motion in two dimensions and let $g$ be a deterministic continuous function, $g: \rset_+ \to \rset^2$.
Let $W_g(t) = \cup_{s\leq t}B(g(s)-\zeta(s),r)$ be a Wiener sausage with drift $g$ up to time $t$. We then have that as $t\to\infty$
\[
\E{\vol{W_g(t)}}\geq \frac{2 \pi t}{\log{t}}(1-o(1)).
\]
\end{lemma}

\begin{proof}
We may write
\[
\E{\text{vol}(W_g(t))} = \int_{\rset^2} \pr{y \in \cup_{s\leq t}B(g(s)-\zeta(s),r)}dy = \int_{\rset^2} \pr{\tau_{B(y,r)}\leq t}dy,
\]
where $\tau_A$ is the first hitting time of the set $A$ by $g-\zeta$.
Define
\[
Z_y = \int_{0}^{t} \1(g(s)-\zeta(s) \in B(y,r)) \,ds,
\]
i.e., the time that the process $g-\zeta$ spends in the ball $B(y,r)$ before time $t$.
It is clear by the continuity of $g-\zeta$ that $\{Z_y > 0 \} = \{\tau_{B(y,r)}\leq t \}$.
Clearly $\pr{Z_y > 0} = \frac{\E{Z_y}}{\Econd{Z_y}{Z_y>0}}$ and for the first moment we have
\begin{align*}
&\E{Z_y} = \int_{0}^{t} \pr{g(s)-\zeta(s) \in B(y,r)} \,ds \\
&= \int_{0}^{t} \int_{B(y,r)} \frac{1}{2\pi s} e^{-\frac{\|z-g(s)\|_2^2}{2s}}\,dz \,ds = \int_{0}^{t} \int_{B(0,r)} \frac{1}{2\pi s} e^{-\frac{\|z+y-g(s)\|_2^2}{2s}} \,dz \,ds.
\end{align*}
For the conditional expectation $\Econd{Z_y}{Z_y>0}$, if we write $\tau$ for the first time
before time $t$ that $g-\zeta$ hits the boundary of the ball $B(y,r)$, denoted by
$\partial B(y,r)$, then we get
\begin{align*}
&\Econd{Z_y}{Z_y>0}= \E{\int_{0}^{t-\tau} \1(g(s+\tau)-\zeta(s+\tau)\in B(y,r))\,ds}
\\
&\leq 1+\E{\int_{1}^{(t-\tau) \vee 1} \1(g(s+\tau)-\zeta(s+\tau)\in B(y,r)) \,ds}
\\ &\leq 1 + \max_{x \in \partial B(y,r)} \int_{1}^{t}  \int_{B(y,r)} \E{\frac{1}{2\pi s} e^{-\frac{\|z-g(s+\tau)-x + g(\tau)\|_2^2}{2s}}} \,dz \,ds
\\ &\leq 1 + \int_{1}^{t}  \int_{B(y,r)}\frac{1}{2\pi s}\,dz\,ds
\leq 1 + r^2 \frac{\log{t}}{2}.
\end{align*}

So, putting everything together we obtain that
\begin{align*}
\E{\text{vol}(W_g(t))} = \int_{\rset^2} \frac{\E{Z_y}}{\Econd{Z_y}{Z_y>0}}\,dy \geq \frac{\int_{0}^{t} \int_{B(0,r)}\left(\int_{\rset^2}\frac{1}{2\pi s}e^{-\frac{\|z+y-g(s)\|_2^2}{2s}} \,dy\right) \,dz \,ds }{1 + r^2\frac{\log{t}}{2}} = \frac{2 \pi t r^2}{2 + r^2\log{t}}
\end{align*}
and hence as $t\to\infty$
\[
\E{\vol{W_g(t)}}\geq \frac{2 \pi t}{\log{t}}(1-o(1)).
\]
\end{proof}

We are now ready to prove Theorem~\ref{thm:detection}.
\begin{proof}[\textbf{Proof of Theorem~\ref{thm:detection}}]
From Remark~\ref{rem:moving} we have
\[
\pr{\tdet > t} = \E{\exp(-\lambda \Econd{\text{vol}(W_g(t))}{(g(s))_{s \leq t}}) },
\]
where $g$ is  independent of $\zeta$. By Lemma~\ref{lem:voldrift}, we have the upper bound
\[
\pr{\tdet > t} \leq \exp{\left(-2\pi \lambda \frac{t}{\log{t}} (1-o(1))\right)}, \text{ as } t \to \infty.
\]
So it remains to show the lower bound on this probability for the case when $g$ is a standard Brownian motion
independent of the motions of the nodes of $\Pi_0$.
Letting $R=\log{t}$, it is clear that
\begin{align}
\pr{\tdet > t} \geq \pr{u \text{ stays in } \B(0,R) \text{ for all } s \leq t, T_{\B(0,R)}>t},
\label{eq:domain}
\end{align}
where $T_{\B(0,R)}$ is the detection time of the ball $\B(0,R)$, i.e.,
\[
 T_{\B(0,R)} = \inf \{s>0: \exists i \text{ s.t. } B(X_i+\zeta_i(s),r)\cap B(0,R)\neq \emptyset\}.
\]
Since the motions of $u$ and the  nodes of $\Pi_0$ are independent, we get that
\begin{equation}
\pr{u \text{ stays in } \B(0,R) \text{ for all } s \leq t, T_{\B(0,R)}>t}
= \pr{u \text{ stays in } \B(0,R) \text{ for all } s \leq t} \pr{ T_{\B(0,R)}>t}.
\label{eq:lb1step}
\end{equation}
From Remark~\ref{rem:comp} with $K=B(0,R)$ we get
\[
\pr{T_{B(0,R)}>t} = \exp(-\lambda \E{\vol{\cup_{s \leq t}B(0,R + r) - \zeta(s)}})
\]
and writing $\cup_{s \leq t}(B(0,R + r) - \zeta(s)) = R \cup_{s \leq t}B\left(-\frac{\zeta(s)}{R},1 + \frac{r}{R}\right) $
and for $R$ large enough we get that,  for all~$s$,
\[
B\left(-\frac{\zeta(s)}{R},1\right)\subset B\left(-\frac{\zeta(s)}{R},1 + \frac{r}{R}\right) \subset B\left(-\frac{\zeta(s)}{R},2\right).
\]
For any $x>0$ we have by Brownian scaling that
\begin{equation*}
\E{\vol{\bigcup_{s \leq t}B\left(-\frac{\zeta(s)}{R},x\right)}}
= \E{\vol{\bigcup_{s' \leq \frac{t}{R^2}}B\left(\tilde{\zeta}(s'),x\right)}},
\end{equation*}
where $\tilde{\zeta}$ is a standard Brownian motion.
So finally, using the asymptotic expression for the expected volume of the Wiener sausage in two dimensions \cite{BMS,Spitzer},
i.e., $\E{\vol{\cup_{s \leq t}B(\zeta(s),x)}} \sim 2 \pi\frac{t}{\log t}$ as $t \to \infty$, for any $x$ independent of $t$, we get that
\begin{equation*}
\E{\vol{\cup_{s \leq t}B(0,R + r) - \zeta(s)}} \sim R^2 2 \pi \frac{t/R^2}{\log t - \log R^2} \sim 2 \pi \frac{t}{\log t}.
\end{equation*}
Hence,
\[
\pr{T_{\B(0,R)}>t} = \exp\left(-2\pi \lambda \frac{t}{\log{t}} (1+o(1))\right), \text{ as } t \to \infty.
\]
Thus we only need to lower bound the probability that $u$ stays in the ball $\B(0,R)$ for all times $s \leq t$.

For any $t \geq R^2$ and any dimension $d\geq 1$,
we have by \cite{CTaylor} that
\begin{equation}
   \pr{u \text{ stays in } B(0,R) \text{ for all } s \leq t} \geq \exp(-ct/R^2),
   \label{eq:taylor}
\end{equation}
for a positive constant $c$ and hence, since $R = \log t$, we get that
\[
\pr{\tdet > t} \geq  \exp\left(-2\pi \lambda \frac{t}{\log{t}}(1+o(1))\right)\exp\left(-c \frac{t}{(\log{t})^2}\right)
= \exp\left(-2\pi \lambda \frac{t}{\log{t}}(1+o(1))\right).
\]
\end{proof}

\topic{General dimensions and the Wiener sausage}
For $d=1$, the volume of $W_g(t)$ can be computed from the maximum and minimum values of $g-\zeta$ via the formula
$$
   \vol{W_g(t)} = 2r + \max_{s \leq t} (g(s)-\zeta(s)) - \min_{s \leq t} (g(s)-\zeta(s)).
$$
Let $t^\star$ and $t_\star$ be the random times in the interval $[0,t]$ at which $-\zeta$ achieves its maximum and minimum values respectively.
Then we have
\begin{align}
   \E{\vol{W_g(t)}}
   &\geq  2r + \E{(g(t^\star)-\zeta(t^\star)) - (g(t_\star)-\zeta(t_\star))}  \nonumber\\
   &= 2r + \E{-\zeta(t^\star)+\zeta(t_\star)} = \E{\vol{W_0(t)}}, \label{eq:gnog1}
\end{align}
where $\E{g(t^\star)}=\E{g(t_\star)}$ holds since $t^\star$ and $t_\star$ have the same distribution.
Thus, for $d=1$ the inequality $\E{\vol{W_g(t)}} \geq \E{\vol{W_0(t)}}$ holds for all fixed $t$;
for $d=2$ Lemma~\ref{lem:voldrift} gives this inequality only asymptotically as $t\to\infty$.

For dimensions $d\geq 3$, the proof of Lemma~\ref{lem:voldrift} can be used to obtain the following weaker result: there exists a positive constant
$c$ such that
\begin{equation}
   \E{\vol{W_g(t)}} \geq c \E{\vol{W_0(t)}}.\label{eq:gnog3}
\end{equation}

The expected volume of the Wiener sausage with $g\equiv 0$ is known to satisfy~\cite{Spitzer,BMS}
\begin{equation}
  V_0(t)= \E{\vol{W_0(t)}} = \left\{\begin{array}{rl}
   \sqrt{\frac{8t}{\pi}}+2r & \text{for $d=1$}\\
   \frac{2\pi t}{\log t}(1+o(1))  & \text{for $d=2$}\\
   c(d)r^{d-2}t(1+o(1))  & \text{for $d\geq 3$,}
   \end{array}\right.
   \label{eq:volwiener}
\end{equation}
where $c(d)=\frac{\Gamma\left(\frac{d}{2}-1\right)}{2\pi^{d/2}}$. (For $d=1$ the quantity above follows from well-known results for
Brownian motion~\cite{BM}; namely $\E{\max_{s\leq t}\zeta(s)} = - \E{\min_{s\leq t}\zeta(s)} = \sqrt{\frac{2t}{\pi}}$.)
Hence plugging (\ref{eq:gnog1}--\ref{eq:volwiener}) into Lemma~\ref{lem:detection} gives an
upper bound for $\pr{\tdet > t}$ when $d=1$ and $d\geq 3$.
Regarding the lower bound, when $g$ is an independent Brownian motion, the same strategy as in the proof of Theorem~\ref{thm:detection}
works for $d=1$ and $d\geq 3$ provided we set $R$ properly.
For $d=1$ it suffices to take $R=t^{1/3}$;
for $d\geq 3$ we can set $R=r$. Then we obtain a positive constant $c_1$ such that, as $t\to\infty$,
$$
   \pr{T_{\B(0,R)}>t} \geq \left\{\begin{array}{rl}
       \exp\left(-\sqrt{\frac{8t}{\pi}}(1+o(1))\right) & \text{for $d=1$}\\
       \exp\left(-c_1 r^{d-2} t (1+o(1))\right) & \text{for $d\geq 3$.}
    \end{array}\right.
$$
This together with \eqref{eq:lb1step} and \eqref{eq:taylor} give us the following theorem, which holds in all dimensions $d\geq 1$.

\begin{theorem}\label{thm:detection2}
   Let $g$ be a continuous process in $\rset^d$, for $d \geq 1$. 
   If $g$ is a deterministic continuous function, then we have
   $$
      -\frac{1}{\lambda}\log \pr{\tdet > t} = \E{\vol{W_g(t)}}.
   $$
   If $g$ is random but independent of the motion of the nodes of $\Pi_0$, we obtain a positive constant $\beta_1$ such that as $t\to\infty$
   $$
      -\frac{1}{\lambda}\log \pr{\tdet > t} \geq \left\{\begin{array}{rl}
         V_0(t) & \text{for $d=1$}\\
         (1-o(1))V_0(t) & \text{for $d=2$} \\
         (1-o(1))\beta_1V_0(t) & \text{for $d\geq 3$},
      \end{array}\right.
   $$
    where $V_0(t)$ is defined in \eqref{eq:volwiener}.

   If $g$ is a standard Brownian motion, then we obtain a positive constant $\beta_2$ such that as $t\to\infty$
   $$
      -\frac{1}{\lambda}\log \pr{\tdet > t} \leq \left\{\begin{array}{rl}
         (1+o(1))V_0(t) & \text{for $d=1,2$}\\
         (1+o(1))\beta_2V_0(t) & \text{for $d\geq 3$}.
      \end{array}\right.
   $$
\end{theorem}

We conclude this section with the following conjecture.
\begin{conjecture}\label{conj:wiener}
   For $d\geq 2$, any fixed $t$, and any function $g$, we have that
   $$
   \E{\vol{W_g(t)}} \geq \E{\vol{W_0(t)}}.
   $$
\end{conjecture}
Note that Theorem~\ref{thm:detection2} establishes the conjecture for $d=1$,
and establishes it asymptotically as $t\to\infty$ for $d=2$ provided $g$ is continuous.
A related statement for random walks was proved in \cite{MOBC};
see also \cite[Corollary~2.1]{DGRS}.

\section{Coverage time}\label{sec:coverage}
In this section we will prove a more general version of Theorem~\ref{thm:coverage}. Let $A$ be a subset of $\R^d$. For $R \in \R_+$ we
define the set $RA = \{Ra: a \in A\}$. We recall the definition of Minkowski dimension, which can be found, e.g., in~\cite{Mattila95}.
\begin{definition}\label{def:mink}
Let $A$ be a non-empty bounded subset of $\R^d$. For $\epsilon>0$ let $M(A,\epsilon)$ be the smallest number of balls of radius $\epsilon$ needed to cover $A$:
\[
M(A,\epsilon) = \min \left\{k: A \subset \bigcup_{i=1}^{k}B(x_i,\epsilon) \text{ for some } x_i \in \R^d\right\}.
\]
The \emph{Minkowski dimension} of $A$ is defined as
\[
\dim_M(A) = \lim_{\epsilon \to 0} \frac{\log M(A,\epsilon)}{\log\epsilon^{-1}},
\]
whenever this limit exists.
\end{definition}
We now proceed to state the more general version of Theorem~\ref{thm:coverage}.
\begin{theorem}\label{thm:coverage2}
Let $A$ be a bounded subset of $\rset^d$ of Minkowski dimension $\alpha$.
We have that as $R \to \infty$
\[
E{\taucov(RA)} \sim \left\{\begin{array}{rl}
   \frac{\alpha^2 \pi}{8\lambda^2}(\log R)^2& \text{for $d=1$}\\
   \frac{\alpha}{2 \pi \lambda }\log{R} \log\log{R}  & \text{for $d=2$}\\
   \frac{\alpha \log R}{\lambda c(d) r^{d-2}}  & \text{for $d\geq 3$}
   \end{array}\right. \quad \text{ and }\quad  \frac{\taucov(RA)}{\E{\taucov(RA)}} \to 1 \text{ in probability},
\]
where $c(d) = \frac{\Gamma\left(\frac{d}{2}-1 \right)}{2 \pi^{\frac{d}{2}}}$ and $\Gamma$~stands for the Gamma function.
\end{theorem}

\begin{proof}
In the proof we will drop the dependence on $RA$ from $\taucov(RA)$ and
$\E{\taucov(RA)}$ to simplify the notation.
We will carry the proof for the case $d=2$ only and discuss how to adapt the proof for other dimensions at the end.

Let $M(A,\epsilon) = \min \{k \geq 1: \exists \quad B_1, \ldots, B_k \text{ balls of radius } \epsilon \text{ covering } A  \}$;
then it is easy to see that $M(RA,\epsilon)=M(A, \frac{\epsilon}{R})$.
By the assumption that $A$ has Minkowski dimension $\alpha$, for any $\delta>0$ we can find $\epsilon_0$ small enough such that
\begin{align} \label{mink}
\epsilon^{-\alpha + \delta} \leq M(A,\epsilon) \leq \epsilon^{-\alpha-\delta}, \text{ for any } \epsilon< \epsilon_0.
\end{align}
We will first show that
\[
\limsup_{R \to \infty} \frac{\tcov}{\frac{\alpha}{2\pi \lambda}\log R \log \log R}\leq 1.
\]
To do so, we are going to cover the set $RA$ by $M=M(RA,\epsilon)$ balls of radius $0<\epsilon<r$.

From \eqref{mink} we get that
\begin{align}\label{mbounds}
\left(\frac{\epsilon}{R} \right)^{-\alpha+ \delta} \leq M \leq \left(\frac{\epsilon}{R}\right)^{-\alpha-\delta}\quad \text{ for $R$ sufficiently large}.
\end{align}
Let $Z_t$ be the number of balls not covered by the nodes at time $t$.
It is clear that $\{\taucov > t \} \subset \{ Z_t \geq 1 \}$. For the first moment of $Z_t$ we have
\[
\E{Z_t} \leq M \pr{\text{a given ball $\B(x,\epsilon)$ is not covered by time $t$}}.
\]
The probability that a ball $\B(x,\epsilon)$ is covered by time $t$ is lower bounded by the probability that a node of the Poisson point process $\Pi_0$ has entered the ball $\B(x,r-\epsilon)$ before time $t$. Hence,
$\pr{B(x,\epsilon) \text{ is not covered by time $t$}}$
is at most the probability that $x$ has not been detected by time $t$ by a mobile geometric graph with radius $r-\epsilon$.
From Lemma~\ref{lem:detection} we obtain
\begin{equation*}
\pr{B(x,\epsilon)\text{ is not covered by time $t$}} \leq e^{-\lambda \E{\text{vol}(W_{0,r-\epsilon}(t))}},
\end{equation*}
where $W_{z,\rho}(t) = \cup_{s \leq t} \B(z+\zeta(s), \rho)$.


We are now prove the \textbf{upper bound}. 
Let $\delta'>0$ be small. For $t$ large enough we have that (see \eqref{eq:volwiener})
\begin{align}\label{wienvol}
(1-\delta')2 \pi\frac{t}{\log{t}} \leq \E{\text{vol}(W_{0,r-\epsilon}(t))} \leq (1+\delta')2 \pi\frac{t}{\log{t}}
\end{align}
and hence,
\begin{align}\label{mesi}
\E{Z_t} \leq M e^{-2 \pi \lambda(1-\delta')\frac{t}{\log{t}}}.
\end{align}
By Markov's inequality we have that $\pr{\taucov > t} \leq \E{Z_t}$. Also,
\begin{align}\label{integral}
\E{\taucov} = \int_{0}^{\infty} \pr{\taucov>t}\,dt \leq t_*(R) + \int_{t_*(R)}^{\infty} M e^{-2 \pi \lambda(1-\delta')\frac{t}{\log{t}}} \,dt,
\end{align}
where $t_*(R)$ satisfies
\begin{align}\label{defin}
M \exp\left(-2\pi \lambda (1-\delta')\frac{t_*(R)}{\log{t_*(R)}}\right) = 1 \text{ and } t_*(R)>e, \text{ for sufficiently large $R$}.
\end{align}
We claim that the last integral appearing in \eqref{integral} is  $o(t_*(R))$.
To see this set $c =2 \pi \lambda (1-\delta') $ and use a change of variable, $x=\frac{t}{\log t}$, which gives $\frac{dx}{dt} \geq \frac{1}{2\log t} \geq \frac{1}{4 \log x}$ for $t$ large enough. Hence setting $x_* = \frac{t_*(R)}{\log t_*(R)}$ the integral is upper bounded by
\begin{equation*}
\int_{x_*}^{\infty} 4 M  (\log{x}) e^{-cx}\,dx \leq c' (1/x_* + \log{x_*}) = o(t_*(R)).
\end{equation*}


So we finally obtain that
\[
\E{\taucov} \leq t_*(R)(1 + o(1)), \text{ as $R \to \infty$}.
\]

From \eqref{defin} and \eqref{mbounds} we get
\[
(\alpha - \delta)\log{\frac{R}{\epsilon}}\leq c \frac{t_*(R)}{\log{t_*(R)}} \leq (\alpha + \delta)\log{\frac{R}{\epsilon}}\quad \text{ for $R$ large enough},
\]
and thus we conclude that
\[
\limsup_{R \to \infty} \frac{\E{\taucov}}{\frac{\alpha}{2 \pi \lambda}\log{R} \log{\log{R}}} \leq 1,
\]
which follows by letting $\delta$ and $\delta'$ go to 0.


We now proceed to show the \textbf{lower bound} 
\[
\liminf_{R \to \infty} \frac{\tcov}{\frac{\alpha}{2\pi \lambda}\log{R} \log{\log{R}}}\geq 1.
\]
To do so, we are going to use the equivalent definition of Minkowski dimension involving packings \cite[Chapter 5]{Mattila95}.
Letting
\[
K(A,\epsilon) = \max\{k \geq 1: \exists \quad B_1, \cdots, B_k \text{ disjoint balls of radius } \epsilon \text{ centered in } A  \},
\]
it is clear that $K(RA,\epsilon) = K(A, \frac{\epsilon}{R})$. For  $\delta>0$ there exist $K=K(RA,1)$ disjoint balls with centers in $RA$ and radius $1$ satisfying
\[
R^{\alpha-\delta} \leq K \leq R^{\alpha+\delta}, \text{ for $R$ large enough}.
\]
So, for $R$ large enough, we can pack the set $RA$ with points $x_1,\cdots,x_K$ (the centers of the balls) that are at distance at least $2$ from each other.
Let $U_t$ denote the number of centers $x_1, \cdots, x_K$ that have not been detected by time $t$.
Obviously we have that $\{\taucov > t \} \supset \{U_t \geq 1 \}$.

Recall that the Wiener sausage $W_{z,r}(t) = \cup_{s \leq t}B(z+\zeta(s),r)$ in two dimensions satisfies, for $\delta'>0$,
\begin{equation}\label{eq:asymp}
(1-\delta')2 \pi\frac{t}{\log{t}} \leq \E{\vol{W_{z,r}(t)}}  \leq (1+\delta')2 \pi\frac{t}{\log{t}}, \text{ for $t$ large enough}.
\end{equation}

Let $\epsilon>0$ be small and let $t^*=t^*(R)> e$ satisfy the equation
\begin{align}\label{tdefin}
\frac{t^*}{\log{t^*}} = \frac{\alpha - \epsilon - \delta}{2 \pi \lambda (1 + \delta')}\log{R}.
\end{align}

Applying the second moment method to the random variable $U_{t^*}$ we obtain
\[
\pr{\taucov > t^*} \geq \frac{(\E{U_{t^*}})^2}{\E{U_{t^*}^2}},
\]
so in order to obtain a lower bound for $\pr{\taucov > t^*}$ it suffices to lower bound the first moment of $U_{t^*}$ and upper bound its second moment. We will show that $\pr{\taucov > t^*} \geq \frac{1}{1+o(1)}$, hence we will get that $\E{\taucov} \geq t^* \frac{1}{1+o(1)}$.

We have that $\E{U_{t^*}} = \sum_{i=1}^{K} \pr{x_i\text{ not detected by time }t^*}$, and using Lemma~\ref{lem:detection} we obtain that
\[
\pr{x_i\text{ not detected by time }t^*} = \exp(- \lambda \E{\text{vol}(W_{x_i})}),
\]
where $W_x=W_{x,r}(t^*) = \bigcup_{s \leq t^*} \B(x+\zeta(s),r)$.

Obviously $\E{\text{vol}(W_x)}$ is independent of $x$, and hence we get that
\begin{align}\label{timi}
\E{U_{t^*}} = K \exp(- \lambda \E{\text{vol}(W_0)}).
\end{align}
Now, for the second moment of $U_{t^*}$ we have
\begin{align}
\label{second}
\E{U_{t^*}^2} =\sum_{i=1}^{K} \sum_{j\neq i} \pr{x_i,x_j \text{ not detected by time $t^*$}} + \E{U_{t^*}}
\end{align}
and using Remark~\ref{rem:comp} we get that
\[
\pr{x_i,x_j \text{ not detected by time $t^*$}} = \exp(-\lambda \E{\text{vol}(W_{x_i} \cup W_{x_j})}).
\]
(Note that the two Wiener sausages $W_{x_i}$ and $W_{x_j}$ use the same driving Brownian motion.)

Writing
\[
\text{vol}(W_{x_i}\cup W_{x_j}) = \text{vol}(W_{x_i}) + \text{vol}(W_{x_j}) - \text{vol}(W_{x_i} \cap W_{x_j}),
\]
equation \eqref{second} becomes
\begin{align}\label{secondnew}
\E{U_{t^*}^2} \leq \exp\left(- 2\lambda \E{\text{vol}(W_0)}\right) \sum_{i=1}^{N} \sum_{j\neq i} \exp(\lambda \E{\text{vol}(W_{x_i} \cap W_{x_j})}) + \E{U_{t^*}}.
\end{align}
Thus it remains to upper bound $\E{\text{vol}(W_{x_i} \cap W_{x_j})}$ for all $i$ and $j$.
If $\|x_i-x_j\|_2 \leq (\log{R})^2$, then we may use the bound $\text{vol}(W_{x_i}  \cap W_{x_j}) \leq \text{vol}(W_{x_i})$.

Recall from \eqref{tdefin} that $t^*(R)=\Theta( \log{R} \log\log{R})$.
The idea is that if $x_i$ and $x_j$ are at distance greater than $(\log{R})^2$ apart, then it is very unlikely that the two sets $W_{x_i}$ and $W_{x_j}$ will intersect.
Specifically, when $\|x_i-x_j\|_2 \geq (\log{R})^2$ it is easy to see that the probability that the two sausages, $W_{x_i}$ and $W_{x_j}$,
intersect is smaller than the probability that a 2-dimensional Brownian motion has traveled distance greater than $\frac{1}{2}(\log{R})^2$ in $t^*$
time steps, and this last probability is bounded above by $ce^{-c (\log{R})^2}$ by the standard bound for the tail of a Gaussian.

When $\|x_i-x_j\|_2 \geq (\log{R})^2$, writing $S_1= \B(x_i,R)$ for the ball of radius $R$ centered at $x_i$ and defining inductively $S_k = \B(x_i,2^{k-1}R) \setminus \B(x_i,2^{k-2}R)$ for all $k \geq 2$, we can split the volume of $W_{x_i} \cap W_{x_j}$ as follows:
\begin{align*}
\E{\text{vol}(W_{x_i} \cap W_{x_j})} = \sum_{n=1}^{\infty} \E{\text{vol}(W_{x_i} \cap W_{x_j} \cap S_n)} \\
\leq c \pi R^2 e^{-c(\log{R})^2} + \sum_{n=2}^{\infty} c' 2^{2n} R^2 e^{-c' 2^{2n}R} \leq c'' R^{-M},
\end{align*}
where $c$, $c'$, $c''$ and $M$ are all positive constants. The first part of the first inequality follows from the discussion above, namely that if the intersection is nonempty, then the Brownian motion must have traveled distance greater than $\frac{1}{2}(\log{R})^2$ in less than $t^*$ steps. If this has happened, then we simply bound the intersection of the two Wiener sausages in the ball $\B(x,R)$ by the volume of the ball. The second part of the first inequality follows by the same type of argument, since now in order to have a nonempty intersection in the set $S_n$, the Brownian motion must have traveled distance at least $2^{n-2}R$ in less than $t^*$ steps, which again is exponentially small.

Finally, the sum appearing in \eqref{secondnew} is bounded above by
\begin{align}\label{integrals}
\sum_{i=1}^{K} \sum_{j \neq i} \1(\|x_i - x_j\|_2 \leq (\log{R})^2) e^{\lambda \E{\text{vol}(W_0)}}
    +  \sum_{i=1}^{K} \sum_{j \neq i} \1(\|x_i - x_j\|_2 > (\log{R})^2) e^{c'' \lambda R^{-M}}.
\end{align}
By \eqref{eq:asymp} and the definition of $t^*$ given in \eqref{tdefin} we get that \eqref{integrals} is bounded from above by
\[
c_1 K (\log{R})^4 R^{\alpha-\delta -\epsilon} + K^2 e^{c'' \lambda R^{-M}}
\]
and hence
\begin{align}\label{tetrag}
\E{U_{t^*}^2} \leq  \exp(-2\lambda \E{\text{vol}(W_0)})\left(c_1 K (\log{R})^4 R^{\alpha-\delta -\epsilon} + K^2 e^{c'' \lambda R^{-M}}\right)
 + \E{U_{t^*}}.
\end{align}
Therefore, putting all the estimates together we get that
\begin{align*}
\pr{U_{t^*}>0} \geq \frac{(\E{U_{t^*}})^2}{\E{U_{t^*}^2}}
\geq \frac{1}{e^{c'' \lambda R^{-M}} + \frac{1}{K}((\log R)^4 R^{\alpha-\delta-\epsilon} + e^{\lambda \E{\vol{W_0}}})}.
\end{align*}
Using the lower bound  $K \geq R^{\alpha - \delta}$, the upper bound for the expected volume from \eqref{eq:asymp} and the definition \eqref{tdefin} of $t^*$,
we deduce that
\[
\pr{U_{t^*}>0} \geq \frac{1}{1+o(1)},
\]
and thus $\tcov \geq \frac{1}{1+o(1)} t^*$.
Since $t^*$ satisfies \eqref{tdefin}, we deduce that
\[
\liminf_{R \to \infty} \frac{\tcov}{\frac{\alpha}{2\pi \lambda }\log{R} \log{\log{R}}} \geq \frac{1-\frac{\epsilon}{\alpha} - \frac{\delta}{\alpha}}{1+\delta'}
\]
and hence, letting $\epsilon$, $\delta$ and $\delta'$ go to 0, we get that
\[
\liminf_{R \to \infty} \frac{\tcov}{\frac{\alpha}{2\pi \lambda }\log{R} \log{\log{R}}} \geq 1.
\]
So, we have shown that
\begin{equation}\label{eq:statem}
\tcov \sim \frac{\alpha}{2\pi \lambda}\log{R} \log{\log{R}}, \text{ as } R \to \infty.
\end{equation}
Now, for $d=2$ it only remains to show the last part of the theorem, namely that $\frac{\taucov}{\tcov}$ converges to 1 in probability as $R \to \infty$.
For any $\gamma>0$ we have that
\begin{align*}
\pr{\left|\frac{\taucov}{\tcov} - 1\right|>\gamma} = \pr{\taucov> (1+\gamma)\tcov } + \pr{\taucov< (1-\gamma)\tcov} \\ \leq \E{Z_{(1+\gamma)\tcov}} + \pr{\taucov< (1-\gamma)\tcov}.
\end{align*}
From \eqref{mesi} and the definition of $M$ we have that
\[
\E{Z_t} \leq R^{\alpha+\delta} \exp\left(-2\pi \lambda (1-\delta')\frac{t}{\log t}\right).
\]
Plugging in $t= (1+\gamma)\tcov$, using \eqref{eq:statem} and taking $\delta'$ sufficiently small gives that
\[
\E{Z_{(1+\gamma)\tcov}} \to 0, \text{ as } R \to \infty.
\]
For $\epsilon, \delta, \delta'$ small enough we get that $(1-\gamma)\tcov < t^*$, so
\[
\pr{\taucov< (1-\gamma)\tcov} \leq \pr{\taucov \leq t^*} \leq \pr{U_{t^*} =0} \leq 1- \frac{(\E{U_{t^*}})^2}{\E{U_{t^*}^2}} = o(1).
\]
Hence we get the desired result that
\[
\pr{\left|\frac{\taucov}{\tcov} - 1\right|>\gamma} \to 0, \text{ as } R \to \infty.
\]

For dimensions $d\neq 2$, the same arguments carry through by employing the proper expression for the expected volume of the Wiener sausage
given in \eqref{eq:volwiener}. Then, we need to set $t_\star(R)$ and $t^\star(R)$ correspondingly. From \eqref{defin} and \eqref{tdefin},
it suffices to set $t_\star$ to satisfy
\begin{align*}
   \exp\left(\lambda (1-\delta') \sqrt{\frac{8t_\star(R)}{\pi}}\right)=M \quad \text{for $d=1$}\\
   \exp\left(\lambda (1-\delta') c(d)r^{d-2}t_\star(R)\right)=M \quad \text{for $d\geq 3$},
\end{align*}
and $t^\star$ to satisfy
\begin{align*}
   \lambda (1+\delta') \sqrt{\frac{8t^\star(R)}{\pi}}=(\alpha-\epsilon-\delta)\log R \quad \text{for $d=1$}\\
   \lambda (1+\delta') c(d)r^{d-2}t^\star(R)=(\alpha-\epsilon-\delta)\log R \quad \text{for $d\geq 3$}.
\end{align*}

\end{proof}

\begin{remark}
\rm{
While the limit defining Minkowski dimension in Definition~\ref{def:mink} may not exist, the corresponding $\limsup$  is denoted by $\overline{\dim}_M(A)$ and always exists.
The proof of Theorem~\ref{thm:coverage2} also shows that for $d=2$
\[
\limsup_{R \to \infty}\frac{\E{\taucov(RA)}}{\frac{1}{2\pi\lambda}\log R \log \log R} = \overline{\dim}_M(A)
\]
and similarly for $\liminf$ and other dimensions.
}
\end{remark}

\begin{remark}\rm{
The estimates in the proof of Theorem~\ref{thm:coverage2} actually imply that a.s.\ $\frac{\taucov(RA)}{\E{\taucov(RA)}} \to 1$ as $R \to \infty$.}
\end{remark}

\section{Percolation time}\label{sec:percolation}
In this section we give the proof of Theorem~\ref{thm:percolation}.
We will observe
the process $(G_i)_{i \geq 0}$ in discrete time steps $i=0, 1, \ldots$ in order to be able to apply a multi-scale argument.
For a nonnegative integer $i$ we define the event $J_i$ that $u$ does not belong to the infinite component at time $i$; more formally,
\[
J_i = \{u \notin \cup_{y \in C_{\infty}(i)}B(y,r) \}.
\]
Then it is easy to see that, for all $t$, we have
\[
\pr{\tperc > t} \leq \pr{\cap_{i=0}^{\lfloor t\rfloor}J_i}.
\]

We define $Q_L$ to be the cube with side length $L$ centered at the origin and with sides parallel to the axes of $\rset^d$.
We tessellate $Q_L$ into subcubes of side length $\ell<L$ , which we call \emph{cells}.
We now state two key propositions that lie at the heart of our argument.

The first proposition says that, provided every cell of the tessellation contains sufficiently many nodes, then we can couple the positions of these nodes
after sufficiently many steps with the nodes of an independent Poisson point process of only slightly smaller intensity on a smaller cube.
We prove this proposition in Section~\ref{sec:coupling}.

\begin{proposition}\label{pro:ss}
   Fix $K > \ell >0$ and consider the cube $Q_K$ tessellated into cells of side length $\ell$.
   Let $\Phi_0$ be an arbitrary point process at time $0$
   that contains at least $\beta \ell^d$ nodes at each cell of the tessellation for some $\beta>0$.
   Let $\Phi_\Delta$ be the point process obtained at time $\Delta$ from $\Phi_0$ after the nodes
   have moved according to standard Brownian motion for time $\Delta$.
   Fix $\epsilon\in(0,1)$ and let $\Xi$ be an independent Poisson point process with intensity $(1-\epsilon)\beta$. Then there exists a coupling of $\Xi$ and $\Phi_{\Delta}$ and constants $c_1,c_2,c_3$ depending only on $d$ such that, if
   $\Delta \geq \frac{c_1 \ell^2}{\epsilon^2}$ and $K' \leq K - c_2  \sqrt{\Delta \log \epsilon^{-1}}>0$, then
   the nodes of $\Xi$ are a subset of the nodes of $\Phi_\Delta$ inside the cube $Q_{K'}$ with probability at least
   $$
      1-\frac{K^d}{\ell^d}\exp(-c_3\epsilon^2\beta\ell^d).
   $$
\end{proposition}

The second proposition, which we prove in Section~\ref{sec:density},
says that the above condition that each cell contains sufficiently many nodes is satisfied at an arbitrary constant fraction of time steps with
high probability.

\begin{proposition}\label{pro:density}
   Let $t>0$ be a sufficiently large integer and $\xi,\epsilon \in (0,1)$ be two constants.
   Suppose that the cube $Q_L$, for $L=t$, is tessellated into cells of side length $\ell$, where
   $\ell^d \geq C \log^3 t$ for some sufficiently large constant $C$.
   For $i=0, 1, \ldots$ let
   \[
   A_i = \{\text{\rm{at time $i$ all cells contain }} \geq (1-\xi)\lambda \ell^d \text{\rm{ nodes of }} \Pi_i \}.
   \]
   Then there exists a positive constant $c$ such that
   \begin{equation}
      \pr{\sum_{i=0}^{t-1} \1(A_i) \geq (1-\epsilon)t} \geq
      1-\exp\left(-c\frac{\lambda t}{\log^{3+6/d} t}\right).
      \label{eq:probdensity}
   \end{equation}
\end{proposition}

\begin{proof}[\textbf{Proof of Theorem~\ref{thm:percolation}}]
Let $u$ be a node that is at the origin at time $0$ independent of the nodes of $\Pi_0$.
We assume that $u$ is non-mobile; the proof can easily be extended to mobile $u$ using translated cubes that track the motion of $u$ as in~\cite[Section~4]{SS10}.

Let $t$ be an integer sufficiently large. We consider the cube $Q_L$, for $L=t$.
Set $H_t$ to be the event that $u$ has \emph{never} been in the infinite component from time
$0$ to $t-1$. More formally, we define
\[
H_t =\cap_{i=0}^{t-1}J_i =  \cap_{i=0}^{t-1}\{u \notin \cup_{y \in C_{\infty}(i)} B(y,r)\}.
\]

We say that a cube $Q_L$ has a \emph{crossing component} at a given time $i$ if among the nodes in $Q_L$
there exists a connected component that has a path connecting each pair of opposite faces of $Q_L$.
(A path \emph{connects} two faces of $Q_L$ if for each face there is
at least one node of the path within distance $r$ of the face.)
We then define $\tilde{H}_t$ to be the event that $u$ has never been within distance $r$ of a crossing component of
$Q_{L}$ from time $0$ to $t-1$.
Let $K_t$ be the event that, in each step from $0$ to $t-1$, there exists a unique crossing component of $Q_L$ and it
intersects the infinite component.
Therefore, if $K_t$ holds and $u$ belongs to a crossing component of $Q_L$ at some time step from $0$ to $t-1$, then at the same time step
$u$ will also belong to the infinite component. We can then conclude that $K_t \cap {\tilde H_t}^\compl \subseteq H_t^\compl$, which gives
$$
   \pr{H_t} \leq \prsmall{K_t^\compl \cup \tilde H_t} \leq \prsmall{K_t^\compl} + \prsmall{\tilde H_t}.
$$
By \cite[Theorems 1 and 2]{PenPis} and by taking the union bound over all time steps, we have
\begin{equation}
   \pr{K_t^\compl}\leq t\exp(-c_1 L).
   \label{eq:kt}
\end{equation}


We will now derive an upper bound for $\prsmall{\tilde H_t}$.
Let $\xi>0$ be a sufficiently small constant such that $(1-\xi)\lambda>\lambda_\critical$.
Take the cube $Q_{2L}$ and tessellate it into cells of side length $\ell$, where $\ell=C_1\log^{3/d}t$,
for $C_1$ a sufficiently large constant in order to satisfy the assumptions of Proposition~\ref{pro:density}.
Call a cell \emph{dense} if it contains more than $(1-\xi)\lambda\ell^d$ nodes.
For $\delta>0$, let $D$ be the event that all cells inside $Q_{2L}$ are dense for at least $(1-\delta)t$ time steps. Applying
Proposition~\ref{pro:density} we obtain a constant $c_2$ such that
$$
   \pr{D} \geq 1-\exp\left(-c_2 \frac{\lambda t}{\log^{3+6/d}t}\right).
$$

We use the event $D$ to obtain an upper bound for $\prsmall{\tilde H_t}$ via
\begin{equation}
   \prsmall{\tilde H_t}
   \leq \prsmall{\tilde H_t \cap D} + \prsmall{D^\compl}
   \leq \prsmall{\tilde H_t \cap D} + \exp\left(-c_2 \frac{\lambda t}{\log^{3+6/d}t}\right).
   \label{eq:d}
\end{equation}
On the event $D$, by definition, we can find a collection $S$ of $(1-\delta)t$ time steps
for which all cells of side length $\ell$ are dense inside the cube $Q_{2L}$.
We set $\Delta = C_2 \ell^2$ for some sufficiently large constant $C_2$.
We define $\tau_1$ as the first time step for which all cells of $Q_{2L}$ are dense.
We now define $\tau_{i+1}$ recursively as the first time step after $\tau_i+\Delta$ for which all cells are dense.
Obviously, $\tau_1 < \tau_2 < \cdots$ and if we take $k=c_3t/\Delta=c_3' t/\log^{6/d}t$ for some constant $c_3$, then
we can ensure that on $D$ we have $\tau_k \leq t-1$.

For each $i$, let $A_i$ be the event that $u$ does
\emph{not} belong to a crossing component of $Q_L$ at time $\tau_i+\Delta$. Since when $D$ holds we
have $\tau_k \leq t-1$, we can write
\begin{equation}
   \prsmall{\tilde H_t \cap D} \leq \pr{\bigcap_{i=1}^k A_i \cap D}.
   \label{eq:addingA}
\end{equation}
For each $i$, let $\mathcal{F}_i$ be the $\sigma$-field induced by the locations of the nodes of $\Pi_0$ from time $0$ to $\tau_i$.
We now claim that for $t$ sufficiently large there exists a positive constant $c_4$ such that
\begin{equation}\label{eq:upper}
\prcond{A_i}{\mathcal{F}_i} < e^{-c_4}.
\end{equation}

We  will define two events $E_1,E_2$ such that for any $F \in \mathcal{F}_i$ we have
\begin{equation}\label{eq:cond}
\prcond{A_i}{F} \leq \prcond{E_1^\compl}{F} + \prcond{E_2^\compl}{F}.
\end{equation}
Take $\epsilon>0$ sufficiently small so that $(1-\epsilon)(1-\xi)\lambda > \lambda_\critical$, and let $\Xi$ be
an independent Poisson point process of intensity $(1-\epsilon)(1-\xi)\lambda$. We define the events
\begin{align*}
&E_1 = \{u \text{ belongs to a crossing component of $\Xi$ in $Q_L$} \} \text{  and}
\\& E_2 = \{\exists \text{ a coupling of } \Xi \text{ and } \Pi_{\tau_i + \Delta} \text{ so that } \Xi \subset \Pi_{\tau_i + \Delta} \text{ in } Q_L \},
\end{align*}
where ``$\Xi \subset \Pi_{\tau_i + \Delta} \text{ in } Q_L$'' means that the nodes of $\Xi$ that lie inside the cube $Q_L$ are a subset of the nodes of
$\Pi_{\tau_i + \Delta}$.

Note that when $E_1$ and $E_2$ both
hold, then $u$ belongs to a crossing component of $Q_L$ at time $\tau_i + \Delta$, which implies that $A_i$ does not hold.
Since the intensity of $\Xi$ is strictly larger than $\lambda_\critical$ and $E_1$ is independent of $F$ by construction,
we obtain $\prcond{E_1}{F} \geq c_5$ for some constant $c_5 \in (0,1)$ by~\cite[Theorem 1]{PenPis}.

All cells are dense at time $\tau_i$, by the definition of $\tau_i$.
Taking $K$ and $K'$ appearing in Proposition~\ref{pro:ss} to be $K=2L$ and $K'=L$,
we see by the choice of $\Delta$ that for large enough $t$ the condition for $K'$
in Proposition~\ref{pro:ss} is satisfied and thus we obtain, uniformly
over all $F \in \mathcal{F}_i$, that, for a positive constant $c_6$,
\[
\prcond{E_2^{\compl}}{F} \leq \exp\left(-c_6 \lambda \log^3t\right).
\]
Plugging everything into \eqref{eq:cond} we get
\[
\prcond{A_i}{F} \leq 1-c_5 + \exp\left(-c_6 \lambda \log^3t\right),
\]
which can be made strictly smaller than $1$ by taking $t$ sufficiently large. This establishes \eqref{eq:upper}.

Note that by definition we have $\tau_i+\Delta < \tau_{i+1}$ for all $i$, which
gives $A_i \in \mathcal{F}_{i+1}$. We can write (\ref{eq:addingA}) as
$$
   \prsmall{\tilde H_t \cap D} \leq \pr{\cap_{i=1}^k A_i} = \prod_{i=2}^k \prcond{A_i}{\cap_{j=1}^{i-1} A_j}\pr{A_1},
$$
which by \eqref{eq:upper} translates to
$$
   \prsmall{\tilde H_t \cap D} \leq \exp\left(-c_4 k\right) \leq \exp\left(-c_7\frac{t}{\log^{6/d}t}\right),
$$
for a positive constant $c_7$.
Plugging this into (\ref{eq:d}) concludes the proof of
Theorem~\ref{thm:percolation}.

\end{proof}

\subsection{Coupling}\label{sec:coupling}
In this section we give the proof of Proposition~\ref{sec:coupling}.
We begin by stating and proving a small technical lemma that will be used in the proof.

\begin{lemma}\label{lem:psi}
Assume $\epsilon \in (0,1)$ and $\rho>0$. Let $\Delta \geq 16 d^2 \rho^2/\epsilon^2$ and $R \geq 2 \sqrt{d \Delta \log (8d \epsilon^{-1})}$.
Define
\[
g(z)= \frac{1}{(2 \pi \Delta)^{d/2}} \exp\left(-\frac{(\|z\|_2 + \rho)^2}{2 \Delta} \right)
\]
on $\R^d$. Then we have
\[
\int_{B(0,R)} g(z) \,dz \geq 1 - \epsilon/2.
\]
\end{lemma}
\begin{proof}
Let $\psi(x) =\frac{1}{(2 \pi \Delta)^{1/2}} \exp\left(-\frac{(|x|+\rho)^2}{2 \Delta} \right) $, for $x \in \R$.
\newline Note that $\sum_{i=1}^{d} (|z_i|+\rho)^2 = \|z\|_2^2 + 2 \rho \|z \|_1 +  \rho^2 d\geq (\|z \|_2 + \rho)^2$,
so
\begin{equation}\label{eq:prod}
\prod_{i=1}^{d} \psi(z_i) \leq g(z), \text{ for } z= (z_1,\ldots,z_d) \in \R^d.
\end{equation}
Next observe that
\begin{equation*}
\int_{-\infty}^{\infty} \psi(x) \,dx = 1 - \int_{-\rho}^{\rho} \frac{1}{(2\pi\Delta)^{1/2}} \exp\left(-\frac{y^2}{2\Delta}\right) \,dy \geq
1- \frac{2\rho}{\sqrt{2\pi \Delta}} \geq 1 - \frac{\rho}{\sqrt{\Delta}} \geq 1 - \frac{\epsilon}{4d}.
\end{equation*}
By the Gaussian tail bound we have that
\[
\int_{R/\sqrt{d}}^{\infty} \psi(x) \,dx \leq \exp\left(-\frac{R^2}{2d \Delta}\right) \leq \frac{\epsilon^{2}}{64 d^2} \leq \frac{\epsilon}{8d},
\]
for any $\epsilon \in (0,1)$. Thus $\int_{-R/\sqrt{d}}^{R/\sqrt{d}} \psi(x) \,dx \geq 1 - \frac{\epsilon}{2d}$.
Since $[-R/\sqrt{d}, R/\sqrt{d}]^d \subset B(0,R)$, we deduce from \eqref{eq:prod} that
\[
\int_{B(0,R)} g(z) \,dz \geq \int_{[-R/\sqrt{d}, R/\sqrt{d}]^d} \prod_{i=1}^{d} \psi(z_i) \,dz \geq \left(1- \frac{\epsilon}{2d}\right)^d \geq 1 - \epsilon/2.
\]
\end{proof}

We now proceed to the proof of Proposition~\ref{sec:coupling}.

\begin{proof}[\textbf{Proof of Proposition~\ref{sec:coupling}}]
We will construct $\Xi$ via three Poisson point processes.
We start by defining $\Xi_0$ as a Poisson point process over $Q_{K}$
with intensity $(1-\epsilon/2)\beta$.
Recall that $\Phi_0$ has at least $\beta \ell^d$ nodes in each cell of $Q_K$.
Then, in any fixed cell,
$\Xi_0$ has fewer nodes than $\Phi_0$ if
$\Xi_0$ has less than $\beta\ell^d$ nodes in that cell, which by a
standard Chernoff bound (cf.\ Lemma~\ref{lem:cbpoisson}) occurs
with probability larger than
$1-\exp\left(-\frac{{\epsilon'}^2(1-\epsilon/2)\beta\ell^d}{2}(1-{\epsilon'}/3)\right)$
for $\epsilon'$ such that $(1+\epsilon')(1-\epsilon/2)=1$. Since $\epsilon \in (0,1)$ we
have $\epsilon'\in(\epsilon/2,1)$, and the probability above can be bounded below by
$1-\exp\left(-c\epsilon^2\beta\ell^d\right)$ for some constant $c=c(d)$.
Let $\{\Xi_0 \preceq \Phi_0\}$ be the event that $\Xi_0$ has fewer nodes
than $\Phi_0$ in every cell of $Q_{K}$.
Using the union bound over cells we obtain
\begin{equation}
   \pr{\Xi_0 \preceq \Phi_0} \geq 1- \frac{K^d}{\ell^d}\exp(-c\epsilon^2\beta\ell^d).
   \label{eq:cbcoupling}
\end{equation}

If $\{\Xi_0 \preceq \Phi_0\}$ holds, then we can map each node of
$\Xi_0$ to a unique node of $\Phi_0$ in the same cell. We will now show that we can couple
the motion of the nodes in $\Xi_0$ with the motion of their respective pairs in
$\Phi_0$ so that the probability that an arbitrary pair is at the same location at time $\Delta$
is sufficiently large.

To describe the coupling, let
$v'$ be a node from $\Xi_0$ located at $y' \in Q_{K}$, and let $v$ be the pair of
$v'$ in $\Phi_0$. Let $y$ be the location of $v$ in $Q_{K}$, and note that since $v$ and
$v'$ belong to the same cell we have $\|y-y'\|_2 \leq \sqrt{d}\ell$.
We will construct a function $g(z)$ that is smaller than the densities
for the motions of $v$ and $v'$ to the location $y'+z$, uniformly for $z\in \mathbb{R}^d$.
That is,
\begin{equation}
   g(z) \leq \frac{1}{(2\pi \Delta)^{d/2}} \exp\left(-\frac{\max\{\|z\|_2^2,\|y'+z-y\|_2^2\}}{2\Delta}\right)
   \label{eq:constraintg}
\end{equation}
for all $z\in \mathbb{R}^d$.

To this end we set
\begin{equation}
   g(z) = \frac{1}{(2\pi \Delta)^{d/2}} \exp\left(-\frac{(\|z\|_2+\sqrt{d}\ell)^2}{2\Delta}\right).
   \label{eq:g1}
\end{equation}
Note that this definition satisfies~(\ref{eq:constraintg}) since by the triangle inequality
$\|y'+z-y\|_2 \leq \|y'-y\|_2+\|z\|_2$ and $\|y'-y\|_2 \leq \sqrt{d}\ell$.
Define $\psi= 1-\int_{\mathbb{R}^d} g(z)\,dz$.
Then, with probability $1-\psi$ we
can use the density function $\frac{g(z)}{1-\psi}$ to sample a single location
for the position of both $v$ and $v'$ at time $\Delta$,
and then set $\Xi'_0$ to be the Poisson point process with
intensity $(1-\psi)(1-\epsilon/2)\beta$
obtained by {\it thinning} $\Xi_0$ (i.e., deleting each node
of $\Xi_0$ with probability $\psi$).
At this step we have crucially used the fact that the function
$g(z)$ in (\ref{eq:g1}) is oblivious of the location of $v$ and, consequently, is independent of the
point process $\Phi_0$. (If one were to use the maximal coupling suggested
by~(\ref{eq:constraintg}), then the thinning probability would depend on $\Phi_0$, and
$\Xi'_0$ would not be a Poisson point process.)

Let $\Xi'_\Delta$ be obtained from $\Xi'_0$
after the nodes have moved
according to the density function $\frac{g(z)}{1-\psi}$.
Thus we are assured that the nodes of the Poisson point process $\Xi'_\Delta$
are a subset of the nodes of $\Phi_\Delta$ and are independent of the nodes of $\Phi_0$, where $\Phi_\Delta$ is
obtained by letting the nodes of $\Phi_0$ move from time $0$ to time $\Delta$.

By Lemma~\ref{lem:psi} we get that if $\Delta$ and $K-K'$ are large enough,
then the integral of $g(z)$ inside the ball $B=B(0,(K-K')/2)$ is larger than
$1-\epsilon/2$. (We are interested in the ball $B$ since for all $z\in Q_{K'}$ we have $z+B\subset Q_K$.)

When
$\{\Xi_0 \preceq \Phi_0\}$ holds,
$\Xi'_{\Delta}$ consists of a subset of the nodes of $\Phi_\Delta$.
Note that $\Xi'_\Delta$ is a \emph{non-homogeneous} Poisson point process over $Q_K$.
It remains to show that the intensity of $\Xi'_\Delta$ is strictly larger than $(1-\epsilon)\beta$ in $Q_{K'}$ so
that $\Xi$ can be obtained from $\Xi'_\Delta$ via thinning; since $\Xi'_\Delta$ is independent of
$\Phi_0$, so is $\Xi$.

For $z \in \mathbb{R}^d$, let $\mu(z)$ be the intensity of $\Xi'_{\Delta}$.
Since $\Xi'_0$ has no node outside $Q_{K}$, we obtain for any $z\in Q_{K'}$,
$$
   \mu(z)
   \geq (1-\psi)(1-\epsilon/2)\beta \int_{z+B} \frac{g(z-x)}{1-\psi} \,dx
   = (1-\epsilon/2)\beta \int_{B}g(x) \,dx,
$$
where the inequality follows since $z+B \subset Q_{K}$ for all $z \in Q_{K'}$.
From Lemma~\ref{lem:psi},
choosing the constants $c_1$ and $c_2$ sufficiently large
 we have
$\int_{B} g(x) \,dx \geq 1-\epsilon/2$. We then obtain
$\mu(z) \geq (1-\epsilon/2)^2\beta \geq (1-\epsilon)\beta$, which is the intensity of $\Xi$.
Therefore, when $\{\Xi_0 \preceq \Phi_0\}$ holds, which occurs with probability given by (\ref{eq:cbcoupling}),
the nodes of $\Xi$ are a subset of the nodes of $\Phi_\Delta$, which completes the proof of Proposition~\ref{pro:ss}.
\end{proof}

\subsection{Density}\label{sec:density}
In this section we prove Proposition~\ref{pro:density} using a multi-scale argument.
Since the argument is rather involved, we begin with a high-level overview.


\topic{Proof overview}
Our goal is to show that if we tessellate the cube $Q_L$, with $L=t$, into cells of volume of order $(\log{t})^c$, then the probability that all cells
contain sufficiently many nodes for a fraction $1-\epsilon$ of the time steps is at least the expression given in Proposition~\ref{pro:density}.

We start at scale $1$ with the cube $Q_{L_1}$ where $L_1 > L$. We tessellate $Q_{L_1}$ into cells that are so large
that we can easily show that with very high probability during \emph{all} time steps all these cells contain sufficiently many nodes.
We refer to this as the event that ``the density condition is satisfied at all steps for scale 1.''
Then, when going from scale $j-1$ to scale $j$,
we take a smaller cube $Q_{L_j}$ with $L_j < L_{j-1}$, and
tessellate it into cells that are smaller than the cells at the previous scale (see Figure~\ref{fig:multiscalecells}).
We define the density condition for scale $j$ at a given time step as the event that all the cells at scale $j$ contain a
number of nodes that is sufficiently large but strictly smaller than the one used for the density condition for scale $j-1$.
Since this density requirement becomes less strict when going from scale $j-1$ to scale $j$, we will be able to show that
the density condition for
scale $j$ is satisfied for a large fraction of the time steps at which the density condition is satisfied for scale $j-1$.
We repeat this procedure until we obtain, at the last scale, the cube $Q_L$ and cells of side length $\ell$.

The importance of the multi-scale approach is that it allows us to recover quickly from instances of low density, i.e., if the density condition
holds in scale $j-1$ but fails (at some time) in scale $j$, there are enough nodes nearby to recover density shortly thereafter.

\begin{figure}[tbp]
   \begin{center}
      \includegraphics[scale=.7]{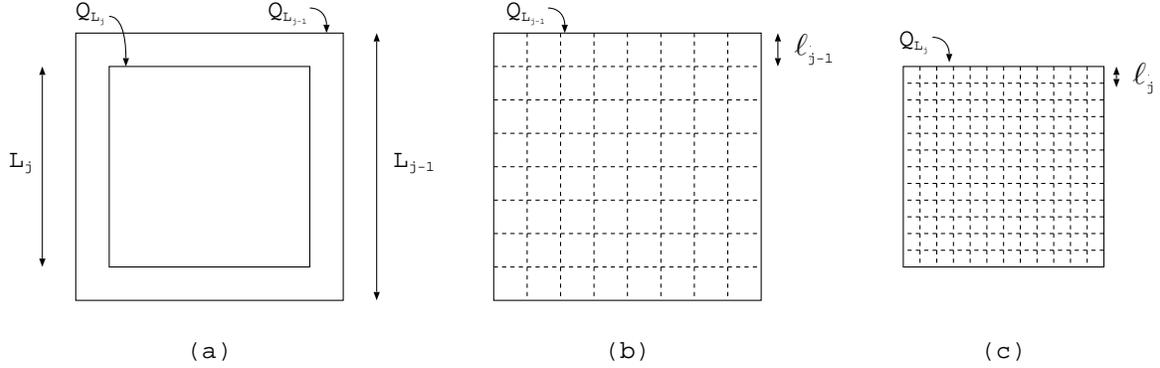}
   \end{center}
   \caption{(a) The cube $Q_{L_{j-1}}$ and the smaller cube $Q_{L_j}$ obtained when going
   from scale $j-1$ to scale $j$. (b) The tessellation of $Q_{L_{j-1}}$ into cells of side length $\ell_{j-1}$.
   (c) The finer tessellation of $Q_{L_{j}}$ into cells of side length $\ell_{j} < \ell_{j-1}$.}
   \label{fig:multiscalecells}
\end{figure}

We now proceed to the detailed argument.

\topic{Full proof}
Let $\kappa$ be the number of scales; we will see in a moment that $\kappa=O(\log t)$ will suffice.
Let $L_1 > L_2 > \cdots > L_\kappa=L$ such that $L_1 = t^2$ and $L_{\kappa}=t$.

Let $\ell_1 > \ell_2 > \cdots > \ell_\kappa =\ell$.
At scale $j$, we consider the cube $Q_{L_j}$ and tessellate it into cells of side length $\ell_j$
(see Figure~\ref{fig:multiscalecells}(b--c)). We
say that a cell is {\emph{dense}} at a given time step for scale $j$ if it contains more than $(1-\xi_j)\lambda\ell_j^d$ nodes
at that step, where the $\xi_j$ satisfy
\[
\frac{\xi}{2}=\xi_1 < \xi_2 < \cdots < \xi_\kappa=\xi \quad \text{ and } \quad \xi_{j}-\xi_{j-1}= \frac{\xi}{2(\kappa-1)}, \text{ for all } j.
\]
We start by analyzing the event that all cells are dense for scale $1$ during all time steps, which we denote by $D_1$.
The next lemma shows that $D_1$ occurs with very high probability.
\begin{lemma}\label{lem:d1}
   If $\ell_1^d > C\log t$ for some large enough constant $C$, then
   there exists a constant $c$ such that
   $$
      \pr{D_1} \geq 1 - \exp\left(-c\lambda \ell_1^d\right).
   $$
\end{lemma}
\begin{proof}
   For any fixed time $i$ and cell $k$, the number of nodes in $k$ at time $i$ is given by a Poisson random variable with
   mean $\lambda \ell_1^d$. Then, using a Chernoff bound (cf. Lemma~\ref{lem:cbpoisson}), we obtain that
   there are more than $(1-\xi_1)\lambda \ell_1^d$ nodes in that cell at that time step with probability
   larger than $1-\exp(-\xi_1^2\lambda\ell_1^d/2)$. The number of cells inside $Q_{L_1}$ is $O(t^{2d})$
   by our choice of $L_1$ and $\ell_1$.
   The proof is completed by taking the union bound over all cells and
   time steps, and using the assumption on $\ell_1$.
\end{proof}

We will need to disregard some time steps when going from one scale to the next.
During this discussion it will be useful to refer to Figure~\ref{fig:multiscale}.
\begin{figure}[tbp]
   \begin{center}
      \includegraphics[scale=.7]{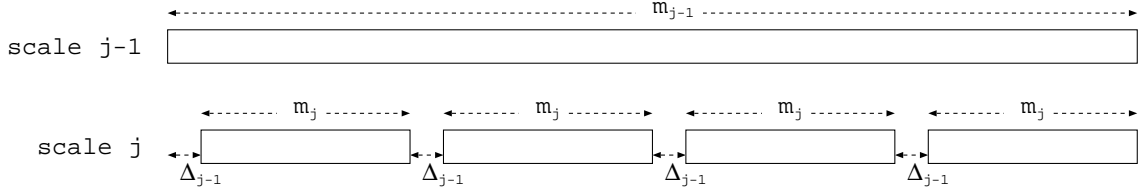}
   \end{center}
   \caption{Illustration for how 1 time interval of scale $j-1$ gives 4 subintervals of scale $j$.}
   \label{fig:multiscale}
\end{figure}
Let $s_j$ be the number of time steps considered for scale $j$.
We start with $s_1=t$ so that at scale $1$ all time steps are considered; we will have
$s_1 > s_2 > \cdots > s_\kappa$.
For each scale $j$, we will split time into intervals of $m_j$ consecutive time steps.
We start with $m_1=t$, so that at scale $1$ we have only one time interval of length $t$.

In each interval $[a,a+m_{j-1})$ at scale $j-1$ we consider the following four separated subintervals of length $m_j$ (see Figure~\ref{fig:multiscale}):
\begin{equation}\label{intervals}
[a+k \Delta_{j-1} + (k-1)m_{j}, a+ k \Delta_{j-1} + k m_{j}),\text{ for } k=1,2,3,4,
\end{equation}
where
\begin{equation}
   m_{j} = \frac{m_{j-1}-4\Delta_{j-1}}{4}.
   \label{eq:mm}
\end{equation}
We will set the $\Delta_{j}$ in a moment.
We skip $\Delta_{j-1}$ steps in order to allow the nodes to move far enough and enable the application of the
coupling from Proposition~\ref{pro:ss}.
Note that this gives $s_j = s_{j-1}\left(1-\frac{4\Delta_{j-1}}{m_{j-1}}\right)$.

For a given scale $j$, we say that a time interval is \emph{dense} if all cells are dense during all the time
steps contained in this time interval, i.e., each cell contains more than $(1-\xi_j)\lambda \ell_j^d$ nodes at all time steps.

Let $0=\epsilon_1 < \epsilon_2 < \cdots < \epsilon_\kappa=\epsilon$ satisfy $\epsilon_j - \epsilon_{j-1} = \frac{\epsilon}{\kappa-1}$.
For each scale $j \geq 1$, we define the event
\begin{align}\label{eq:defeven}
D_j = \{\text{a fraction of at least $\geq\left(1-\frac{\epsilon_j}{2}\right)$ time intervals of scale $j$ are dense}\}.
\end{align}

If $D_{\kappa}$ holds, the number of time steps for which all cells are dense for the last scale $\kappa$ is at least
\begin{equation}\label{eq:sdef}
   \left(1-\frac{\epsilon_\kappa}{2}\right)s_\kappa
   = \left(1-\frac{\epsilon_\kappa}{2}\right)s_1\prod_{j=1}^{\kappa-1} \left(1-\frac{4\Delta_{j}}{m_{j}}\right)
   \geq \left(1-\frac{\epsilon}{2}\right)t\left(1-\sum_{j=1}^{\kappa-1}\frac{4\Delta_{j}}{m_{j}}\right).
\end{equation}
Since we are aiming to obtain $(1-\epsilon)t$ time steps for which the density condition is satisfied
for the last scale, we set $\Delta_j$ to satisfy
\begin{equation}
   \frac{\Delta_j}{m_j} = \frac{\epsilon}{8\kappa}
   \label{eq:dm}
\end{equation}
for all $j$.
The value of $\Delta_j$ must be sufficiently large to allow nodes to move over a distance $\ell_{j}$.
We then define $\ell_j$ by
\begin{equation}
   \Delta_{j} = C'\ell_j^2 \kappa^2,
   \label{eq:dl}
\end{equation}
where $C'$ is a sufficiently large constant.

From (\ref{eq:mm}), (\ref{eq:dm}) and (\ref{eq:dl}),
we obtain
\begin{equation}
   \frac{\ell_j^2}{m_j}=\frac{\epsilon}{8C'\kappa^3} \quad \text{ and } \quad \ell_{j+1} = \ell_j \sqrt{\frac{1}{4} - \frac{\epsilon}{8\kappa}}.
   \label{eq:ellm}
\end{equation}
Since $m_1=t$, we get that $\ell_1^2 = \frac{\epsilon}{8C'\kappa^3}t \leq \frac{\epsilon}{8C'}t$ and since we want to get $\ell_\kappa^d = \ell^d \geq C (\log t)^{3}$,
it is easy to see that $\kappa = O(\log t)$  is sufficient.


For any time step $i$, let $\mathcal{F}_{i}$ be the $\sigma$-field induced by the locations of the nodes of $\Pi_0$
from time $0$ up to time $i$.


\begin{lemma}\label{lem:scalemarkov}
   Let $A=[a,a+m_{j})$ be a time interval considered in scale $j$. We write $b=a-\Delta_{j-1}$ and $E = \{\text{at time $b$ all cells are dense for scale $j-1$} \}$.
   Let $\ell^d \geq C (\log t)^3$ for some sufficiently large constant $C>0$.
   Then there exists a constant $c$ such that
   $$
      \prcond{A \text{ not dense},E}{\mathcal{F}_{b}}
      \leq \exp\left(-c\lambda \ell_j^d/\kappa^2\right).
   $$
\end{lemma}

\begin{proof}
   For any $F\in\mathcal{F}_{b}$ such that $F \cap E =\emptyset$ the lemma clearly holds.
   We then take $F \cap E \neq \emptyset$ and give an upper bound for
   $\prcond{A \text{ not dense}}{E,F}$.
   Let $\Phi_b$ be the point process obtained at time $b$ after conditioning on $F \cap E$.
   We first fix a time $w \in A$ and derive an
   upper bound for
   \[
   \prcond{\text{at time $w$ not all cells are dense for scale $j$}}{E, F}.
   \]
  Since we condition on $E$, all cells are dense for scale $j-1$ at time $b$.
  We now set $\delta$ such that $(1-\delta)^2(1-\xi_{j-1})=1-\xi_j$, which implies $\delta = \Theta(\xi_j-\xi_{j-1})$.
  We also choose a constant $c$ and the constant $C'$ appearing in the definition of $\Delta_j$ in \eqref{eq:dl} so that, setting
   \begin{equation}\label{defL}
   L_j \leq L_{j-1} - c \sqrt{\Delta_{j-1} \log{\frac{1}{\delta}}},
   \end{equation}
   allows us to apply Proposition~\ref{pro:ss} with $K=L_{j-1}$ and $K'=L_j$.
   Thus we obtain a fresh Poisson point process $\Xi$ with intensity $(1-\delta)(1-\xi_{j-1})\lambda$
   that can be coupled with $\Phi_w$
   (which is the point process obtained at time $w$ after the points of $\Phi_b$ have moved for time $w-b$)
   in such a way that $\Xi$ is stochastically
   dominated by $\Phi_w$ inside $Q_{L_j}$ with probability at least
   \begin{equation}
      1-\exp\left(-c_1\delta^2(1-\xi_{j-1})\lambda\ell_{j-1}^d\right),
      \label{eq:ssa}
   \end{equation}
   for some positive constant $c_1$.
   We note that the choice of $L_1=t^2$ and the fact that $\kappa=O(\log{t})$ together with equation \eqref{eq:dl} gives
   that it is always possible to choose the $L_j$'s satisfying \eqref{defL} and such that $L_{\kappa}=t$.

   A given cell is dense for scale $j$ at time $w$
   if $\Xi$ contains at least $(1-\xi_j)\lambda\ell_j^d$ nodes in that cell,
   which by the choice of $\delta$ happens with probability at least
   $1-\exp\left(-c_2\delta^2(1-\delta)(1-\xi_{j-1})\lambda \ell_j^d\right)$ for some constant $c_2$ (cf. Lemma~\ref{lem:cbpoisson}).
   The proof is completed by taking the union bound over all cells and over all time steps in $A$ and using the condition for $\ell$.
\end{proof}

We now use Lemma~\ref{lem:scalemarkov} to give an upper bound for $\pr{D_{j}^\compl \cap D_{j-1}}$ that holds for all $j$,
where $D_j$ was defined in \eqref{eq:defeven}.
\begin{lemma}\label{lem:itscale}
   If $\ell^d \geq C(\log t)^3$ for some large enough $C$, then
   there exists a constant $c$ such that for any $j\geq 2$ we have
   $$
      \pr{D_{j}^\compl \cap D_{j-1}}
      \leq \exp\left(-\frac{c\lambda t(\log t)^{3-6/d}}{\kappa^6} \right).
   $$
\end{lemma}
\begin{proof}
   If $D_{j-1}$ happens, then there are at least $\left(1-\frac{\epsilon_{j-1}}{2}\right)\frac{s_{j-1}}{m_{j-1}}$ dense time intervals for scale $j-1$.
   When we go to scale $j$, these intervals will give us
   \begin{equation}
     4\left(1-\frac{\epsilon_{j-1}}{2}\right)\frac{s_{j-1}}{m_{j-1}}
      \label{eq:goodparent}
   \end{equation}
   time intervals that we will consider for scale $j$.
   On the other hand, if the event $D_{j}^\compl$ holds, then
   there are less than
   \begin{equation}
      \left(1-\frac{\epsilon_{j}}{2}\right)\frac{s_{j}}{m_{j}}
      \label{eq:goodchild}
   \end{equation}
   dense intervals for scale $j$.
   Let $w$ be obtained by subtracting (\ref{eq:goodchild}) from (\ref{eq:goodparent}), that is,
   \begin{eqnarray*}
      w= \frac{s_j}{m_j} \left(\frac{\epsilon_j- \epsilon_{j-1}}{2}\right).
   \end{eqnarray*}
   Let $Z$ be the number of subintervals $[a,a+m_j)$ of scale $j$ that are not dense for scale $j$, but
   are such that the time step $a-\Delta_{j-1}$ is dense for scale $j-1$. (We call a time step dense
   if all cells are dense at that time.)
   It is easy to see that if both $D_{j-1}$ and $D_{j}^\compl$ happen, then $Z \geq w$.

   We can write $Z$ as a sum of $s_{j}/m_{j}$ indicator random variables $I_k$, one for each time interval of scale $j$.
   Although the $I_k$'s depend on one another, Lemma~\ref{lem:scalemarkov} gives that
   the probability that $I_k=1$ given an arbitrary realization of the previous $k-1$ indicators
   is smaller than $\rho_{j}=\exp\left(-c_1\xi^2\lambda \ell_{j}^d/\kappa^2\right)$ for some constant $c_1$.
   Therefore, $Z$ is stochastically
   dominated by a random variable $Z'$ obtained as a
   sum of $s_{j}/m_{j}$ i.i.d. Bernoulli random variables with mean $\rho_{j}$. Using a Chernoff bound
   (cf. Lemma~\ref{lem:cbbinomial}),
   we obtain
   \begin{eqnarray}
      \pr{Z' \geq w}
      &=& \pr{Z' - \E{Z'} \geq \frac{s_j}{m_j} \left( \frac{\epsilon_j-\epsilon_{j-1}}{2} - \rho_j \right) } \nonumber
     \\& \leq &\exp\left(-\frac{s_j}{m_j} \left(\frac{\epsilon_j - \epsilon_{j-1}}{2} \right)\left(\log \left(\frac{\epsilon_j - \epsilon_{j-1}}{2\rho_j} \right)-1 \right)  \right).
         \label{eq:chernoff}
   \end{eqnarray}
   Note that $\epsilon_{j}-\epsilon_{j-1}=\frac{\epsilon}{\kappa-1}$ and $\log(\rho_{j}^{-1})=\Theta(\xi^2\lambda\ell_{j}^d/\kappa^2)$. Also $\ell_j^d \geq \ell^d \geq C(\log t)^3$ and $\kappa = O(\log t)$, so we obtain a constant $c_2$ such that
   \begin{equation*}
      \pr{Z' \geq w} \leq \exp\left(-c_2 \xi^2 \lambda s_j \frac{\ell_j^d}{m_j}\frac{\epsilon}{\kappa^3}  \right).
   \end{equation*}
   Recall from \eqref{eq:ellm} that $\frac{\ell_{j}^2}{m_{j}}=\frac{\epsilon}{8C'\kappa^3}$. By \eqref{eq:sdef} and \eqref{eq:dm} we have that $s_{j-1}=\Theta(t)$ for all $j$, so we finally obtain
   $$
      \pr{Z' \geq w}
      \leq \exp\left(
         -c_3\frac{\epsilon^2\xi^2}{\kappa^6}\lambda\ell_{j}^{d-2} t\right)
   $$
   for some constant $c_3>0$. Using $\ell_j \geq \ell$ and the assumption on $\ell$ in the statement of the lemma completes the proof.
\end{proof}

We are now in a position to prove Proposition~\ref{pro:density}.
\begin{proof}[\textbf{Proof of Proposition~\ref{pro:density}}]
   To prove Proposition~\ref{pro:density}, we need to derive an upper bound for $\pr{D_\kappa^\compl}$. Note that
   $\pr{D_\kappa^\compl} \leq \pr{D_\kappa^\compl \cap D_{\kappa-1}} + \pr{D_{\kappa-1}^\compl}$. Applying
   this inequality recursively for the term $\pr{D_{\kappa-1}^\compl}$ we obtain
   $$
      \pr{D_\kappa^\compl}
      \leq \sum_{j=2}^{\kappa}\pr{D_{j}^\compl \cap D_{j-1}} + \pr{D_1^\compl}.
   $$
   Each term in the sum can be bounded using Lemma~\ref{lem:itscale} and the last term can be bounded using
   Lemma~\ref{lem:d1}. The proof is completed since $\kappa=O(\log t)$ and the initial value
   \[
   \ell_1 = \frac{\epsilon}{8C'\kappa^3}t \geq c_1\frac{t}{(\log{t})^{3}}.
   \]
\end{proof}

\section{Broadcast time}\label{sec:broadcast}
In this section we use Theorem~\ref{thm:percolation} to prove Corollary~\ref{cor:broadcast} for
a finite mobile network of volume $n/\lambda$.

We may relate the mobile geometric graph model on the torus to a model on~$\rset^d$ as follows.
Let $S_n$ denote the cube $Q_{(n/\lambda)^{1/d}}$.
The initial distribution of the nodes is a Poisson point process over~$\rset^d$
with intensity~$\lambda$ on~$S_n$ and zero elsewhere.
We allow the nodes to move according to Brownian motion over~$\rset^d$
as usual, and at each time step we project the location of each node onto~$S_n$
so that nodes ``wrap around'' $S_n$ when they reach the boundary.

\begin{proof}[\textbf{Proof of Corollary~\ref{cor:broadcast}}]
   Let $t= C \log n(\log\log n)^{3+6/d}$ for some sufficiently large constant $C=C(d)$.
   We define a \emph{giant component} as a connected component that
   contains at least two nodes at distance larger than
   $\frac{(n/\lambda)^{1/d}}{4}$. It follows from~\cite[Theorem~2]{PenPis} and the union bound
   over time steps that, with probability $1-e^{-\Theta(n^{1/d})}$,
   $G_i$ contains a \emph{unique} giant component
   for all integer $i\in[0,2t-1]$.

   The proof proceeds in two stages.  First, we show that for any fixed $i\in[0,2t-1]$,
   w.h.p. the giant component of $G_i$ has at least one node in common with the giant
   component of~$G_{i+1}$. This means that, once the message has reached the giant
   component, it will reach any node~$v$ as soon as~$v$ itself belongs to the giant
   component.
   Then we show that, after $t$ steps, all nodes have belonged to the giant component w.h.p.
   This implies that broadcast is achieved after $2t$ steps w.h.p.

   To establish the first stage, let $\epsilon>0$ be sufficiently small so
   that $(1-\epsilon)\lambda>\lambda_\critical$. We use
   the thinning property to split $\Pi_i$ into two Poisson point processes,
   $\Pi_i'$ and~$\Pi_i''$, with intensities $(1-\epsilon)\lambda$ and $\epsilon \lambda$
   respectively. Let $G_i'$ and $G_{i+1}'$ be the graphs induced by~$\Pi'_i$ and $\Pi'_{i+1}$
   respectively.  Then with probability $1-e^{-\Theta(n^{1/d})}$ both $G_i'$ and $G_{i+1}'$
   contain a unique giant component~\cite[Theorem~2]{PenPis}.   We show that
   at least one node from $\Pi_i''$ belongs to both giant components.
   For any node $v$ of $\Pi_i''$, the probability that $v$
   belongs to the giant component of $G_i'$ is larger than some constant $c=c(d)$.
   Moreover, using the FKG inequality we can show that $v$ belongs to the giant
   components of both $G_i'$ and $G_{i+1}'$ with probability larger than~$c^2$.
   Therefore, using the thinning property again, we can show that the nodes from~$\Pi_i''$
   that belong to the giant components of both $G_i'$ and $G_{i+1}'$ form a Poisson point
   process with intensity $\epsilon \lambda c^2$, since $c$ does not depend on $\Pi_i''$.
   Hence, there will be at
   least one such node inside~$S_n$ with probability $1-e^{-\epsilon c^2 n}$, and this
   stage is concluded by taking the union bound over time steps~$i$.

   We now proceed to the second stage of the proof. We first need to show that the tail bound on $\tperc$
   from Theorem~\ref{thm:percolation} also holds when applied to the finite region~$S_n$
   defined above.  Note that all the derivations in the proof of Theorem~\ref{thm:percolation}
   were restricted to the cube~$Q_{L_1}$, where $L_1=t^2$ was defined in Section~\ref{sec:density}.
   We have that $Q_{L_1}$ is contained inside $S_n$ for all sufficiently large $n$
   since $L_1=t^2=O(\log^2n(\log\log n)^{6+12/d})$ while $S_n$ has side length
   $(n/\lambda)^{1/d}$. In order to check that the toroidal boundary conditions do not
   affect the result, it suffices to observe that, during the time interval
   $[0,2t]$, no node moved distance larger than $\frac{(n/\lambda)^{1/d}}{2}$ w.h.p.


   Now note that, by a Chernoff bound, $G$ has at most $(1+\delta)n$ nodes with probability
   larger than $1-e^{-\Omega(n)}$ for any fixed $\delta>0$.  These nodes are indistinguishable,
   so letting $\rho$ be the probability that an arbitrary node has percolation time at
   least~$t$, we can use the union bound to deduce that this applies to {\it at least one\/}
   node in $G$ with probability at most $(1+\delta)n\rho$.  Let $v$ be an arbitrary node.
   In order to relate $\rho$ to the result of Theorem~\ref{thm:percolation},
   we can use translation invariance and assume that $v$ is at the origin.
   Then, by the ``Palm theory" of Poisson point processes~\cite{SKM95}, $\rho$~is equivalent
   to the tail of the percolation time for a node added at the origin, which is precisely
   $\pr{\tperc> t}$.  Thus finally, using Theorem~\ref{thm:percolation} we get
   $\rho \leq \exp(-c \frac{t}{(\log t)^{3+6/d}})$, which can be made $o(1/n)$
   by setting $C$ sufficiently large in the definition of $t$.

   We then obtain that with probability $1-o(1/n)$
   all nodes of $G$ have been in the giant component during the time interval $[0,t-1]$, which implies that at
   time step $t-1$, the nodes of the giant component contain the message being broadcast. By stationarity,
   with probability $1-o(1/n)$ all nodes have been in the giant component during the time interval $[t,2t-1]$, and thus
   have received the message by time $2t$. This completes the proof of Corollary~\ref{cor:broadcast}.
\end{proof}

\begin{remark}\rm{
It is easy to see that the above result also holds in the case where the graph has
\emph{exactly} $n$ nodes.
The proof above shows that, by setting $C$ large enough, we can ensure
$\pr{\tbroad > 2t}=o(1/n)$ for the given value of~$t$.
Also, it is well known that a Poisson random variable with mean~$n$ takes the value~$n$
with probability $p=\Theta(1/\sqrt{n})$. Therefore, for a graph with exactly $n$ nodes,
we have $\Pr[\tbroad < t]=\frac{p-o(1/n)}{p}=1-o(1/\sqrt{n})$.
}\end{remark}


\section*{Acknowledgments}
We are grateful to Takis~Konstantopoulos and David~Tse for useful discussions.

\bibliographystyle{plain}
\bibliography{tperc}

\appendix

\section{Standard large deviation results}

We use the following standard Chernoff bounds and large deviation results.

\begin{lemma}[Chernoff bound for Poisson]\label{lem:cbpoisson}
   Let $P$ be a Poisson random variable with mean $\lambda$. Then, for any
   $0<\epsilon<1$,
   $$
      \pr{P \geq (1+\epsilon) \lambda} \leq  \exp\left(-\frac{\lambda \epsilon^2}{2}(1-\epsilon/3)\right),
   $$
   and
   $$
      \pr{P \leq (1-\epsilon) \lambda} \leq  \exp\left(-\frac{\lambda \epsilon^2}{2}\right).
   $$
\end{lemma}


\begin{lemma}[Chernoff bound for binomial~{\cite[Corollary~A.1.10]{AlonSpencer}}]\label{lem:cbbinomial}
   Let $X$ be the sum of $n$ i.i.d. Bernoulli random variables with mean $p$. Then,
   $\pr{X \geq np+a} \leq \exp\left(a-(pn+a)\log\left(1+\frac{a}{pn}\right)\right)$.
\end{lemma}

\end{document}